\documentclass[1 [leqno,11pt]{amsart}
\usepackage{amssymb, amsmath}
\def\ccite#1{~\cite{#1}}

\def\longformule#1#2{
\displaylines{ \qquad{#1} \hfill\cr \hfill {#2} \qquad\cr } }
\def\inte#1{
\displaystyle\mathop{#1\kern0pt}^\circ }

\def\orr{\omega^r}
\def\ot{\omega^\theta}

\def\ur{u^r}
\def\ut{u^\theta}

\let\e=\varepsilon

\let\d=\partial
\let\pa=\partial
\let\wt=\widetilde

\def\cF{{\mathcal F}}

\def\la{\lambda}

\let\f=\frac
\let\D=\Delta

\renewcommand{\div}{{\rm div}\,}

\newcommand{\rmnum}[1]{\romannumeral #1}
\newcommand{\Rmnum}[1]{\uppercase\expandafter{\romannumeral #1} }
 \numberwithin{equation}{section}

\let\s=\sigma
\let\th=\theta

\def\virgp{\raise 2pt\hbox{,}}
\def\cdotpv{\raise 2pt\hbox{;}}

\def\eqdefa{\buildrel\hbox{\footnotesize def}\over =}

\def\C{\mathop{\mathbb C\kern 0pt}\nolimits}
\def\DD{\mathop{\mathbb D\kern 0pt}\nolimits}
\def\EE{\mathop{{\mathbb E \kern 0pt}}\nolimits}
\def\K{\mathop{\mathbb K\kern 0pt}\nolimits}
\def\N{\mathop{\mathbb N\kern 0pt}\nolimits}
\def\Q{\mathop{\mathbb Q\kern 0pt}\nolimits}
\def\R{\mathop{\mathbb R\kern 0pt}\nolimits}
\def\SS{\mathop{\mathbb S\kern 0pt}\nolimits}
\def\ZZ{\mathop{\mathbb Z\kern 0pt}\nolimits}
\def\TT{\mathop{\mathbb T\kern 0pt}\nolimits}
\def\PP{\mathop{\mathbb P\kern 0pt}\nolimits}

\def\dive{\mathop{\rm div}\nolimits}
\def\curl{\mathop{\rm curl}\nolimits}

\def\no{\noindent}
\def\na{\nabla}

\def\p3{\partial_3}

\def\u3{u^3}

\def\b3{b^3}

\def\om3{\omega^3}

\newcommand{\beq}{\begin{equation}}
\newcommand{\eeq}{\end{equation}}
\newcommand{\ben}{\begin{eqnarray}}
\newcommand{\een}{\end{eqnarray}}
\newcommand{\beno}{\begin{eqnarray*}}
\newcommand{\eeno}{\end{eqnarray*}}
\newcommand{\andf}{\quad\hbox{and}\quad}
\newcommand{\with}{\quad\hbox{with}\quad}

\newtheorem{thm}{Theorem}[section]
\newtheorem{lem}{Lemma}[section]
\newtheorem{rmk}{Remark}[section]
\newtheorem{cor}{Corollary}[section]
\newtheorem{prop}{Proposition}[section]

\begin{document}
\title[Global solutions to 3-D axi-symmetric Navier-Stokes system]
{On the global well-posedness of 3-D axi-symmetric Navier-Stokes system with small swirl component }
\author[Y. Liu]{Yanlin Liu}
\address[Y. Liu]{department of mathematical sciences, university of science and technology of china, hefei 230026, china,
and Academy of Mathematics $\&$ Systems Science, Chinese Academy of
Sciences, Beijing 100190, CHINA.} \email{liuyanlin3.14@126.com}
\author[P. Zhang]{Ping Zhang} \address[P. Zhang]{Academy of Mathematics $\&$ Systems Science
and  Hua Loo-Keng Key Laboratory of Mathematics, Chinese Academy of
Sciences, Beijing 100190, CHINA, and School of Mathematical Sciences, University of Chinese Academy of Sciences, Beijing 100049, China.} \email{zp@amss.ac.cn}

\date{\today}

\maketitle

\begin{abstract} In this paper, we prove the local well-posedness of 3-D axi-symmetric Navier-Stokes system with initial data in
the critical Lebesgue spaces. We also obtain the global well-posedness result with small initial data. Furthermore, with the initial swirl
component of the velocity being sufficiently small in the almost critical spaces, we can still prove the global well-posedness of the system.
\end{abstract}



\section{Introduction}\label{sec1}
In this paper, we investigate the well-posedness and long-time  behavior
of global solutions to 3D axisymmetric  Navier-Stokes equations with a small swirl component.
In general, 3-D  Navier-Stokes system  in $\R^3$ reads
\begin{equation}\label{1.1}
\left\{
\begin{array}{l}
\displaystyle \d_t u+u\cdot\nabla u-\Delta u+\nabla p=0,   \qquad(t,x)\in\R^+\times\R^3\\
\displaystyle \div u=0,\\
\displaystyle u|_{t=0} =u_0,
\end{array}
\right.
\end{equation}
where $u(t,x)=(u^1,u^2,u^3)$  stands for the velocity field and $p$
the scalar pressure function of the fluid, which guarantees the divergence free condition of the velocity field.
 This system describes the motion of viscous incompressible fluid flows.

\smallbreak

We  recall that except the initial data with special structure,
it is not known whether or not the system \eqref{1.1} has a
unique global smooth solution with large smooth initial data. For
instance,  the system \eqref{1.1}  is globally well-posed for data
which is axisymmetric and without swirl component (that is the case when
$u^\th=0$ in \eqref{1.2} below). In this case, Ladyzhenskaya \cite{La}  and independently Ukhovskii and Yudovich
\cite{UY}  proved the
existence of weak solutions along with the uniqueness and regularities of such solution for \eqref{1.1}.
Leonardi, M\'{a}lek, Ne$\breve{c}$as and Pokorny \cite{LMNP} gave a
refined proof of the same result in \cite{ La, UY}. And even with a small swirl component, the authors \cite{ZZT2} could also
establish the global well-posedess of \eqref{1.1}. In general, even the
global wellposedness of \eqref{1.1} with axisymmetric initial data
is still open.

\smallbreak

On the other hand, in the seminal paper \ccite{Leray},    Leray proved the global existence
  of finite energy weak solutions to \eqref{1.1}. Yet the uniqueness and regularity to
this weak solution are big open questions in the field of
mathematical fluid mechanics. Furthermore, Leray emphasized two facts about Navier-Stokes system.
Firstly, he pointed out that energy estimate method is very important to study navier-Stokes system.
The general energy inequality for \eqref{1.1}
\beno \frac
12 \|v(t)\|_{L^2}^2  +\int_0^t\|\nabla v(t')\|_{L^2}^2dt' = \frac 12
\|v_0\|_{L^2}^2, \eeno
is the cornerstone of the proof
to the existence of global turbulent solution to \eqref{1.1}  in\ccite{Leray}.  The energy estimate  relies (formally) on
the fact that if~$v$ is a divergence free vector field,~$ (v\cdot
\nabla f|f)_{L^2} =0~$ and that~$ (\nabla p|v)_{L^2}=0$. In the
present work, we shall use the more general fact that for any
divergence free vector field~$v$ and any function~$a$, we have
$$
\int_{\R^3} v(x) \cdot \nabla a(x) |a(x)|^{p-2}a(x)\, dx =
0\qquad\mbox{for any}\ \ p\in ]1,\infty[.
$$
This will lead to the ~$L^p$ type energy  estimate. Secondly Leray pointed out that
the scaling invariance of \eqref{1.1}, that is,
\beq\label{scale}
v(t,x)\mapsto \la v(\la^2 t, \la x) \andf p(t,x)\mapsto \la^2 p((\la^2 t, \la x),
\eeq
if~$(v,p)$ is a
solution of \eqref{1.1} on~$[0,T]\times \R^3$ associated with an initial
data~$v_0$, then~$(v_\la,p_\la)$ is also a solution of
\eqref{1.1} on~$[0,\lambda^{-2}T]\times \R^3$ associated with the initial
data~$\lambda v_0(\lambda x),$  is another important fact in the study of Navier-Stokes system. The
scaling property is also the foundation of the Kato theory  which
gives a general method to solve (locally or globally) the
incompressible Navier-Stokes equation  in critical spaces i.e.
spaces with the norms of which are invariant under the scaling. In
what follows, we shall use  such scaling invariant space as $L^\infty(]0,t[; L^1(\Omega)),$ where the norm $L^1(\Omega)$ is given by
\eqref{lpo}.

In fact, Gally and $\breve{S}$ver\'ak \cite{GS} recently proved the global well-posedness of 3-D axisymmetric Navier-Stokes system
without swirl and with initial data in the scaling invariant function spaces. We remark that the reason why one can prove the global well-posedness
of \eqref{1.1} in this case is due to the $\theta$ component of the vorticity, $\omega^\theta,$ satisfies
\beno
\pa_t \f{\ot}{r}+(u^r\pa_r+u^z\pa_z) \f{\ot}{r}-(\Delta+\frac 2r\pa_r)\f{\ot}{r}
=0.
\eeno
The scaling invariant Lebesgue space for $\frac{\ot}r$ is $L^\infty(]0,t[;L^1(\R^3)).$ Motivated by \cite{GS},
the purpose of this paper is to improve the norm for the initial data in \cite{ZZT2} to be scaling invariant ones.
We remark that the other motivation of this paper comes from \cite{CZ4} where the authors proved that one scaling invariant
norm to one component of Navier-Stokes system controls the regularity of the solution. Yet we still do not know in general
the global well-posedness of Naver-Stokes with one component being small in some scaling invariant space.

\smallbreak Now we restrict ourselves to the axisymmetric solutions
of  \eqref{1.1} with the following  form
$$u(t,x)=u^r(t,r,z)e_r+u^\theta(t,r,z)e_\theta+u^z(t,r,z)e_z,$$
where $(r,\theta,z)$ denotes the usual cylindrical coordinates in $\R^3$
so that $x=(r\cos\theta,r\sin\theta,z)$,  and
$$e_r=(\cos\theta,\sin\theta,0),\ e_\theta=(-\sin\theta,\cos\theta,0),\ e_z=(0,0,1),\ r=\sqrt{x_1^2+x_2^2}.$$
Then in this case, we can reformulate \eqref{1.1} as
\begin{equation}\label{1.2}
\left\{
\begin{array}{l}
\displaystyle \pa_t u^r+(u^r\pa_r+u^z\pa_z) u^r-(\pa_r^2+\pa_z^2+\frac 1r\pa_r-\frac{1}{r^2})u^r-\frac{(u^\theta)^2}{r}+\pa_r p=0,\\
\displaystyle \pa_t \ut+(u^r\pa_r+u^z\pa_z) \ut-(\pa_r^2+\pa_z^2+\frac 1r\pa_r-\frac{1}{r^2})u^\theta+\frac{u^r u^\theta}{r}=0,\\
\displaystyle \pa_t u^z+(u^r\pa_r+u^z\pa_z) u^z-(\pa_r^2+\pa_z^2+\frac 1r\pa_r)u^z+\pa_z p=0,\\
\displaystyle \pa_r u^r+\frac 1r u^r+\pa_z u^z=0,\\
\displaystyle u|_{t=0} =u_0.
\end{array}
\right.
\end{equation}

Let us denote $\widetilde{u}\eqdefa u^r e_r+u^z e_z.$ Then it is easy to check that
\beno
\dive \wt{u}=0 \andf \curl \wt{u} =\omega^\theta e_\theta,
\eeno so that the Biot-Savart law shows that $u^r$ and $u^z$ can be uniquely determined by $\ot$ (see Subsection \ref{sec2.1}).
Hence we can write
 the System \eqref{1.2} as
\beq\label{otut}
\left\{
\begin{array}{l}
\displaystyle \pa_t\ot-\bigl(\pa_r^2+\pa_z^2+\frac 1r \pa_r-\frac{1}{r^2}\bigr)\ot=-\div_*(\widetilde{u}\ot)+\f{2u^\theta\pa_zu^\theta}r, \\
\displaystyle \pa_t\ut-\bigl(\pa_r^2+\pa_z^2+\frac 1r \pa_r-\frac{1}{r^2}\bigr)\ut=-\div_*(\widetilde{u}\ut)-\frac{2\ut\ur}{r},\\
\displaystyle (\ot,\ut)|_{t=0} =(\ot_0, \ut_0).
\end{array}
\right.
\eeq
Here and all in that follows, we always denote $\div_* f\eqdefa\pa_r f^r+\pa_z f^z$ and abuse the notation $\tilde{u}=(u^r,u^z).$

\smallbreak As in \cite{GS}, we shall equip the half-plane $\Omega=\{(r,z)|r>0,z\in\R\}$
with the 2D measure $drdz$, instead of the 3D measure $rdrdz.$ For any $p\in[1,\infty[$, we denote by $L^p(\Omega)$
the space of measurable functions $f:\Omega\rightarrow\R$ which verifies
\beq \label{lpo} \|f\|_{L^p(\Omega)}\eqdefa\bigl(\int_\Omega|f(r,z)|^p drdz\bigr)^{\frac 1p}<\infty,\quad 1\leq p<\infty. \eeq
The space $L^\infty(\Omega)$ can be defined with the usual modification. Sometimes, we shall also use the 3D Lebesgue measure
$rdrdz$, and the corresponding Lebesgue spaces are then denoted by $L^p(\R^3)$ or $L^p$ with norm
$$\|g\|_{L^p}\eqdefa\bigl(\int_\Omega|f(r,z)|^p rdrdz\bigr)^{\frac 1p},\ 1\leq p<\infty.$$

Our main results state as follows.

\begin{thm}\label{thmlocal}
{\sl For any initial data $\ot_0\in L^1(\Omega)$ and $\ut_0\in L^2(\Omega)$
satisfying $r^{-\frac{3}{10}}\ut_0\in L^{\frac{20}{13}}(\Omega)$, there exists
some $T(\ot_0,\ut_0)$ such that
the equations \eqref{otut} have a unique mild solution
\begin{equation}\begin{split}\label{thmlocal1}
\ot\in C\bigl([0,T];&L^1(\Omega)\bigr)\bigcap C\bigl(]0,T];L^\infty(\Omega)\bigr),\
\ut\in C\bigl([0,T];L^2(\Omega)\bigr)\bigcap C\bigl(]0,T];L^\infty(\Omega)\bigr)\\
&\with\quad r^{-\frac{3}{10}}{\ut}\in C\bigl([0,T];L^{\frac{20}{13}}(\Omega)\bigr)\bigcap C\bigl(]0,T];L^\infty(\Omega)\bigr).
\end{split}\end{equation}
Furthermore, the solution $(\ot, \ut)$ verifies
\begin{itemize}
\item
for any $p\in[1,\infty],~q\in[2,\infty]$ and $\kappa\in[{20}/{13},\infty],$ there holds
\begin{equation}\begin{split}\label{LMN1}
L_p(T)\eqdefa\sup_{0\leq t\leq T}t^{1-\frac 1p}&\|\ot(t)\|_{L^p(\Omega)}<\infty,\quad
M_q(T)\eqdefa\sup_{0\leq t\leq T}t^{\frac12-\frac 1q}\|\ut(t)\|_{L^q(\Omega)}<\infty,\\
&N_\kappa(T)\eqdefa\sup_{0\leq t\leq T}t^{\frac{13}{20}-\frac 1{\kappa}}\|r^{-\frac{3}{10}}{\ut}(t)\|_{L^\kappa(\Omega)}<\infty.
\end{split}\end{equation}
Moreover, when $p\in]1,\infty],~q\in]2,\infty]$ and $\kappa\in]{20}/{13},\infty],$ we have
\begin{equation}\label{thmlocal3}
\lim_{t\rightarrow 0}\bigl( L_p(t)+M_q(t)+N_\kappa(t)\bigr)=0;
\end{equation}

\item
 if \beq\label{smallc}
  \|\ot_0\|_{L^1(\Omega)}+\|\ut_0\|_{L^2(\Omega)}
  +\|r^{-\frac{3}{10}}\ut_0\|_{L^{\frac{20}{13}}(\Omega)}\leq c \eeq
  for some sufficiently
small constant $c,$ then  $T=\infty.$
And if $\|\ut_0\|_{L^2(\Omega)}+\|r^{-\frac{3}{10}}\ut_0\|_{L^{\frac{20}{13}}(\Omega)}$ is small enough,
then the lifespan $T^\star$ of the solution depends only on $\ot_0$.

\end{itemize}
}
\end{thm}

\begin{rmk}\label{rmk1.1}
\begin{itemize}
\item Let us remark that the norms $\|\ot_0\|_{L^1(\Omega)},~\|\ut_0\|_{L^2(\Omega)}$ and
$\|r^{-\frac{3}{10}}\ut_0\|_{L^{\frac{20}{13}}(\Omega)}$ are scaling invariant
under the scaling transformation \eqref{scale}. Moreover,
the method used here might be used to
study axi-symmetric vortex ring for 3-D Navier-Stokes system with swirl (see the corresponding result of
\cite{FS15} for the case without swirl).

\item The reason for requiring $r^{-\frac{3}{10}}\ut_0\in L^{\frac{20}{13}}(\Omega)$ is to handle the term
$\frac{\pa_z|\ut(s)|^2}{r}$ in $\ot$ equation of \eqref{inteeqt}, so that the exponent $\frac 32-\frac 1p+\frac15$ appearing in \eqref{fix6}
is less than $1.$
\end{itemize}
\end{rmk}

\begin{thm}\label{thmmain}
{\sl Let $\ot_0$ and $\ut_0$ satisfy $\eta_0\eqdefa\frac{\ot_0}{r}\in L^1,
~U_0\eqdefa\frac{\ut_0}{r}\in L^{\frac{3}{2}}$,
~$r\ut_0\in L^{A}\bigcap L^\infty$ for some finite $A$. We assume that
$\|r\ut_0\|_{L^\infty(\Omega)}$ is sufficiently
small,
then the system \eqref{otut} has a unique global solution which satisfies
\begin{equation}
\eta\eqdefa \f{\ot}{r}\in C\bigl([0,+\infty[;L^1\bigr) \andf U\eqdefa \f{\ut}r\in C\bigl([0,+\infty[;L^{\frac32}\bigr).
\end{equation}
}
\end{thm}

\begin{rmk} \begin{itemize}
\item The main difficulty in the proof of the above theorem is when $\ot_0\in L^p(\Omega)$ for $p=1,$
the dissipative term, $\frac{4(p-1)}{p^2}\|\nabla|\eta|^{\frac p2}\|_{L^2}^2,$ in \eqref{3.19} disappears. That is the reason
why we divide the proof of Theorem \ref{thmmain} in the following two steps: we first get, by applying Theorem \ref{thmlocal}, that the system
\eqref{otut} has a unique local solution with $\eta(t_0)\in L^{p_0}(\R^3)$ for some $t_0>0$ and $p_0>1;$ then in the second
step, starting with initial data at $t_0,$ we prove the global well-posedness of the system \eqref{otut}.

\item One may see \eqref{5.6},~\eqref{5.19},~\eqref{5.21} and \eqref{5.27}
for the exact smallness condition for $\|r \ut_0\|_{L^\infty}$.
And the exact global estimate of $\|\eta(t)\|_{L^1}$ and $\|U(t)\|_{L^{\f32}}$ is given in \eqref{5.28}.

\item It follows from  Lemma \ref{lem2.1.2} below and H\"{o}lder's inequality
$$\|r^\kappa\ut\|_{L^{\frac{3}{1-\kappa}}}\leq\|U\|_{L^{\frac32}}^{\frac{1-\kappa}2}
\|r\ut\|_{L^\infty}^{\frac{1+\kappa}2},\quad \forall\ \kappa\in]-1,1[,$$
 that the solutions constructed in Theorem \ref{thmmain} in fact satisfy
\begin{equation*}
r^\kappa\ut\in C\bigl([0,+\infty[;L^{\frac{3}{1-\kappa}}\bigr),\quad \forall\ \kappa\in[-1,1].
\end{equation*}
\end{itemize}
\end{rmk}

\section{Preliminaries}\label{sec2}
\subsection{Some elementary results}\label{sec2.1}
\begin{lem}[Proposition 1 of \cite{CL02}]\label{lem2.1.2}
{\sl Let $(u^r,u^\theta, u^z)$ be a smooth enough solution of \eqref{1.2} on $[0,T].$ Then for any $p\in[2,\infty]$, we have
\begin{equation}\label{2.1.1}
\|r\ut(t)\|_{L^p}\leqslant\|r\ut_0\|_{L^p}\quad \forall \ t\in [0,T].
\end{equation}
}
\end{lem}

\begin{lem}[See Lemma 5.5 from \cite{BCD} for instance]\label{lem4.5}
{\sl Let $E$ be a Banach space, $\frak{B}(\cdot,\cdot)$ a continuous bilinear map from $E\times E$ to $E,$ and $\frak{\alpha}$ a positive
real number such that
\beno
\frak{\alpha}<\f1{4\|\frak{B}\|}\with \|\frak{B}\|\eqdefa\sup_{\|f\|,\|g\|\leq 1}\|\frak{B}(f,g)\|.
\eeno
Then for any $a$ in the ball $B(0,\frak{\alpha})$ in $E,$  there exists a unique $x$ in $B(0,2\frak{\alpha})$ such that
\beno
x=a+\frak{B}(x,x).
\eeno}
\end{lem}

Let us recall also some facts from Section 2 of \cite{GS}. We first recall the axisymmetric Biot-Savart law
which determines $\widetilde{u}=(u^r,u^z)$ in terms of $\omega^\theta,$ namely
\begin{equation}\label{2.2.1}
u^r(r,z)=\int_\Omega G_r(r,z,\bar{r},\bar{z})\ot(\bar{r},\bar{z})d\bar{r}d\bar{z},\quad
u^z(r,z)=\int_\Omega G_z(r,z,\bar{r},\bar{z})\ot(\bar{r},\bar{z})d\bar{r}d\bar{z},
\end{equation}
where
\beq\label{2.2.2}
\begin{split}
G_r(r,z,\bar{r},\bar{z})=&-\frac{1}{\pi}\frac{z-\bar{z}}{r^{3/2}\bar{r}^{1/2}}F'(\xi^2),\quad
\xi^2=\frac{(r-\bar{r})^2+(z-\bar{z})^2}{r\bar{r}},\\
G_z(r,z,\bar{r},\bar{z})=&\frac{1}{\pi}\frac{r-\bar{r}}{r^{3/2}\bar{r}^{1/2}}F'(\xi^2)
+\frac{1}{4\pi}\frac{\bar{r}^{1/2}}{r^{3/2}}\bigl(F(\xi^2)-2\xi^2 F'(\xi^2)\bigr)
\with\\
F(s)=& \int_0^{\frac{\pi}{2}}\frac{\cos(2\phi)d\phi}{(\sin^2\phi+s/4)^{1/2}},\quad s>0.
\end{split}
\eeq
It follows from the Remark 2.2 of \cite{GS} that
\begin{lem}\label{lem2.2.1}
{\sl $s^\alpha F(s)$ and $s^\beta F'(s)$ are bounded on $]0,\infty[$ for
$\alpha\in ]0,3/2]$ and $\beta\in[1,5/2]$.
}
\end{lem}

\begin{lem}[Proposition 2.3 of \cite{GS}]\label{lem2.2.2}
{\sl Let us denote $\widetilde{u}\eqdefa(u^r,u^z).$ Then one has\\
$\rmnum{1})$ Assume that $1<p<2<q<\infty$ and $\frac 1q=\frac 1p-\frac 12$.
If $\ot\in L^p(\Omega)$, then $\widetilde{u}\in L^q(\Omega)$ and
\beq\label{2.2.5}
\|\widetilde{u}\|_{L^q(\Omega)}\leqslant C\|\ot\|_{L^p(\Omega)}.
\eeq
$\rmnum{2})$ If $1\leqslant p<2<q\leqslant\infty$ and $\ot\in L^p(\Omega)\bigcap L^q(\Omega)$,
then $\widetilde{u}\in L^\infty(\Omega)$ and
\beq\label{2.2.6}
\|\widetilde{u}\|_{L^\infty(\Omega)}\leqslant C\|\ot\|_{L^p(\Omega)}^\sigma\|\ot\|_{L^q(\Omega)}^{1-\sigma},
\quad\hbox{where}\quad \sigma=\frac{p(q-2)}{2(q-p)}\in ]0,1[.
\eeq
}
\end{lem}

Next we investigate the solution operator $S(t)$ to the linearized system of \eqref{otut}, namely
$\omega^\th(t)=S(t)\omega_0$ verifies
\beq\label{linear}
\left\{
\begin{array}{l}
\displaystyle \pa_t\ot-\bigl(\pa_r^2+\pa_z^2+\frac 1r \pa_r-\frac{1}{r^2}\bigr)\ot=0,   \quad(t,r,z)\in\R^+\times\Omega\\
\displaystyle \ot|_{r=0}=0,\\
\displaystyle \ot|_{t=0} =\ot_0.
\end{array}
\right.
\eeq

\begin{lem}[Lemma 3.1,~3.2 of \cite{GS}]
{\sl For any $t>0$, one has
\begin{equation}\label{lemexpresion1}
\bigl(S(t)\omega_0\bigr)(r,z)=\frac{1}{4\pi t}\int_{\Omega}\frac{\bar{r}^{1/2}}{r^{1/2}}
{H}\Bigl(\frac{t}{r\bar{r}}\Bigr)\exp\Bigl(-\frac{(r-\bar{r})^2+(z-\bar{z})^2}{4t}\Bigr)
\omega_0(\bar{r},\bar{z})d\bar{r}d\bar{z},
\end{equation}
where the function ${H}:]0,+\infty[\rightarrow\R$ is defined by
\begin{equation}\label{lemexpresion2}
{H}(t)=\frac{1}{\sqrt{\pi t}}\int_{-\pi/2}^{\pi/2}e^{-\frac{\sin^2 \phi}{t}}\cos(2\phi) d\phi,\qquad t>0,
\end{equation}
which is smooth on $]0,\infty[$ and has the asymptotic expansions:\\
\rmnum{1})~${H}(t)=\frac{\pi^{1/2}}{4t^{3/2}}+\mathcal{O}\bigl(\frac{1}{t^{5/2}}\bigr),\
{H}'(t)=-\frac{3\pi^{1/2}}{8t^{5/2}}+\mathcal{O}\bigl(\frac{1}{t^{7/2}}\bigr),\ \mbox{as $t\rightarrow\infty$;}$\\
\rmnum{2})~${H}(t)=1-\frac{3t}{4}+\mathcal{O}(t^2),\
{H}'(t)=-\frac{3}{4}+\mathcal{O}(t),\ \mbox{as $t\rightarrow0$.}$
}
\end{lem}

\begin{cor}\label{corH}
$t^\alpha H(t)$ and $t^\beta H'(t)$
are bounded on $]0,\infty[$  provided
$0\leq\alpha\leq\frac 32$,~$0\leq\beta\leq\frac 52$.
\end{cor}

\subsection{The estimate of $\f{u^r}r$ in terms of $\f{\omega^\th}r$}

In this subsection, we shall exploit the basic facts recalled in Subsection \ref{sec2.1}  to derive the estimate of $\f{u^r}r$ in terms of $\f{\omega^\th}r$,
which will be used in Section \ref{sec4} below.
The main result states as follows:

\begin{prop}\label{lem2.6.1}
{\sl Let  $p\in]1,3[$ and  $q\in\bigl]\frac{3p}{3-p},\infty\bigr].$ We assume that
$\eta\eqdefa\frac{\ot}{r}\in L^p(\R^3)\cap L^{3p}(\R^3)$. Then  we have
\begin{equation}\label{3.1}
\bigl\|\frac{u^r}{r}\bigr\|_{L^q}\lesssim\|\eta\|_{L^p}^{\lambda}\|\eta\|_{L^{3p}}^{1-\lambda}
\with \quad \lambda=\frac{p-1}{2}+\frac{3p}{2q}.
\end{equation}}
\end{prop}

\begin{proof}
 By virtue of  \eqref{2.2.1} and \eqref{2.2.2}, we  write
\begin{equation}\label{3.2}
r^{\frac 1q}\frac{\ur(r,z)}{r}=-\int_\Omega\frac{z-\bar{z}}{\pi}\frac{\bar{r}^{\frac 12}}{r^{\frac 52-\frac 1q}}
F'(\xi^2)\eta(\bar{r},\bar{z})d\bar{r} d\bar{z}\quad \forall \ (r,z)\in\Omega.
\end{equation}
We  decompose the integral domain $\Omega=I_1\bigcup I_2$  with
\begin{equation}\label{3.3}
I_1\eqdefa \bigl\{(\bar{r},\bar{z})\in\Omega\ \big|\ \bar{r}\leq 2r,\  \bigr\} \andf
I_2\eqdefa\Omega\backslash I_1.
\end{equation} We first consider the case when $q<\infty$.
Let  $s$ be determined by
$\frac 1s=\frac 1q+\frac 13.$ Then due to  $q\in\bigl]\frac{3p}{3-p},\infty\bigr[,$  we have  $s>p.$ Moreover, it follows from  Lemma \ref{lem2.2.1} that
 $|F'(s)|\lesssim(\frac{1}{s})^{\frac 76}=\left(\frac{1}{s}\right)^{1+\frac12\left(\frac 1s-\frac 1q\right)}.$
Note that that $\frac{\bar{r}}{r}\leqslant 2$ in $I_1$ and
$\frac32-\frac12(\frac 1s+\frac 1q)>0,$  we thus obtain
\begin{equation*}\begin{split}
&\bigl|\int_{I_1}\frac{z-\bar{z}}{\pi}\frac{\bar{r}^{\frac 12}}{r^{\frac 52-\frac 1q}}
F'(\xi^2)\eta(\bar{r},\bar{z})d\bar{r} d\bar{z}\bigr|\\
&\lesssim\int_{I_1}\frac{|z-\bar{z}|\cdot\bar{r}^{\frac 12-\frac 1s}}{r^{{\frac52}-\frac 1q}}
\left(\frac{r\bar{r}}{(r-\bar{r})^2+(z-\bar{z})^2}\right)^{\frac76}
\cdot\bar{r}^{\frac 1s}|\eta(\bar{r},\bar{z})|d\bar{r} d\bar{z}\\
&\lesssim\int_{I_1}\left(\frac{\bar{r}}{r}\right)^{\frac32-\frac12(\frac 1s+\frac 1q)}
\left(\frac {1}{|(\bar{r},\bar{z})-(r,z)|}\right)^{\frac 43}\cdot\bar{r}^{\frac 1s}|\eta(\bar{r},\bar{z})|d\bar{r} d\bar{z}\\
&\lesssim\int_{\Omega}
\left(\frac {1}{|(\bar{r},\bar{z})-(r,z)|}\right)^{\frac 43}\cdot\bar{r}^{\frac 1s}|\eta(\bar{r},\bar{z})|d\bar{r} d\bar{z},
\end{split}\end{equation*}
from which, and Hardy-Littlewood-Sobolev inequality, we infer
\begin{equation}\label{3.4}\begin{split}
&\bigl\|\int_{I_1}\frac{z-\bar{z}}{\pi}\frac{\bar{r}^{\frac 12}}{r^{\frac 52-\frac 1q}}
F'(\xi^2)\eta(\bar{r},\bar{z})d\bar{r} d\bar{z}\bigr\|_{L^q(\Omega;drdz)}\\
&\lesssim \bigl\||(r,z)|^{-\frac 43}\bigr\|_{L^{\frac32,\infty}(\Omega)}  \|r^{\frac 1s}\eta\|_{L^{s}(\Omega)}\\
&\thicksim\|\eta\|_{L^{s}}\lesssim\|\eta\|_{L^p}^{\frac{p-1}{2}+\frac{3p}{2q}}\|\eta\|_{L^{3p}}^{1-\frac{p-1}{2}-\frac{3p}{2q}}.
\end{split}\end{equation}

Note that in the region $I_2,$ there holds $\bar{r}\leq 2|\bar{r}-r|$.
Thus by using
Lemma \ref{lem2.2.1}, we get
\begin{equation*}\begin{split}
\Bigl|\int_{I_{2}}\frac{z-\bar{z}}{\pi}\frac{\bar{r}^{\frac 12}}{r^{\frac 52-\frac 1q}}
F'(\xi^2)\eta(\bar{r},\bar{z})d\bar{r} d\bar{z}\Bigr|&\lesssim\int_{I_2}\frac{|z-\bar{z}|\cdot\bar{r}^{\frac 12}}{r^{{\frac52}-\frac 1q}}
\left(\frac{r\bar{r}}{(r-\bar{r})^2+(z-\bar{z})^2}\right)^{\frac52-\frac 1q}
\cdot|\eta(\bar{r},\bar{z})|d\bar{r} d\bar{z}\\
&\lesssim \int_{I_{2}}\left(\frac {1}{|(\bar{r},\bar{z})-(r,z)|}\right)^{1-\frac 1q}
|\eta(\bar{r},\bar{z})|d\bar{r} d\bar{z}.
\end{split}\end{equation*}

To proceed further, for any given $R>0,$ we split $I_2=I_{21}\cup I_{22}$ with
$$I_{21}=I_2\cap\bigl\{(\bar{r},\bar{z})\in\Omega\big||(\bar{r},\bar{z})-(r,z)|\geq R\bigr\}, \quad
I_{22}=I_2\cap\bigl\{(\bar{r},\bar{z})\in\Omega\big||(\bar{r},\bar{z})-(r,z)|< R\bigr\}.$$

Then we get, by applying Young's inequality, that
\begin{equation*}\begin{split}
&\bigl\|\int_{I_{21}}\frac{z-\bar{z}}{\pi}\frac{\bar{r}^{\frac 12}}{r^{\frac 52-\frac 1q}}
F'(\xi^2)\eta(\bar{r},\bar{z})d\bar{r} d\bar{z}\bigr\|_{L^q(\Omega)}\\
&\leq\bigl\|\int_{I_{21}}\left(\frac {1}{|(\bar{r},\bar{z})-(r,z)|}\right)^{1+\frac{1}{p}-\frac 1q}
\cdot\bar{r}^{\frac{1}{p}}|\eta(\bar{r},\bar{z})|d\bar{r} d\bar{z}\bigr\|_{L^q(\Omega)}\\
&\lesssim\bigl\|r^{\frac{1}{p}}\eta\bigr\|_{L^{p}(\Omega)}\Bigl(\int_R^\infty\rho^{\frac{2/q-{2}/{p}}{1-{1}/{p}+1/q}}
\rho d\rho\Bigr)^{1-\frac{1}{p}+\frac 1q}\thicksim R^{1-\frac 3p+\frac 3q}\|\eta\|_{L^{p}}.
\end{split}\end{equation*}

For the integral on $I_{22}$,
in the case $q>3p$, by applying Young's inequality, we get
\begin{equation*}\begin{split}
&\bigl\|\int_{I_{22}}\frac{z-\bar{z}}{\pi}\frac{\bar{r}^{\frac 12}}{r^{\frac 52-\frac 1q}}
F'(\xi^2)\eta(\bar{r},\bar{z})d\bar{r} d\bar{z}\bigr\|_{L^q(\Omega)}\\
&\leq\bigl\|\int_{I_{22}}\left(\frac {1}{|(\bar{r},\bar{z})-(r,z)|}\right)^{1+\frac{1}{3p}-\frac 1q}
\cdot\bar{r}^{\frac{1}{3p}}|\eta(\bar{r},\bar{z})|d\bar{r} d\bar{z}\bigr\|_{L^q(\Omega)}\\
&\lesssim\bigl\|r^{\frac{1}{3p}}\eta\bigr\|_{L^{3p}(\Omega)}\Bigl(\int_0^R\rho^{\frac{2/q-{2}/{3p}}{1-{1}/{3p}+1/q}}
 d\rho\Bigr)^{1-\frac{1}{3p}+\frac 1q}\thicksim R^{1-\frac 1p+\frac 3q}\|\eta\|_{L^{3p}}.
\end{split}\end{equation*}
While in the case $\frac{3p}{3-p}<q\leq 3p$, another use of Young's inequality gives
\begin{equation*}\begin{split}
&\bigl\|\int_{I_{22}}\frac{z-\bar{z}}{\pi}\frac{\bar{r}^{\frac 12}}{r^{\frac 52-\frac 1q}}
F'(\xi^2)\eta(\bar{r},\bar{z})d\bar{r} d\bar{z}\bigr\|_{L^q(\Omega)}\\
&\leq\bigl\|\int_{I_{22}}\left(\frac {1}{|(\bar{r},\bar{z})-(r,z)|}\right)^{1+\frac{3-p}{3p}-\frac 1q}
\cdot\bar{r}^{\frac{3-p}{3p}}|\eta(\bar{r},\bar{z})|d\bar{r} d\bar{z}\bigr\|_{L^q(\Omega)}\\
&\lesssim\bigl\|r^{\frac{3-p}{3p}}\eta\bigr\|_{L^{\frac{3p}{3-p}}(\Omega)}\Bigl(\int_0^R\rho^{\frac{2/q-{2(3-p)}/{3p}}{1-{(3-p)}/{3p}+1/q}}
 d\rho\Bigr)^{1-\frac{3-p}{3p}+\frac 1q}\\
&\thicksim R^{2-\frac 3p+\frac 3q}\|\eta\|_{L^{\frac{3p}{3-p}}}
\lesssim R^{1-\frac 3p+\frac 3q}\|\eta\|_{L^{p}}+R^{1-\frac 1p+\frac 3q}\|\eta\|_{L^{3p}}.
\end{split}\end{equation*}
As a result, it comes out
\beno
&\bigl\|\int_{I_{2}}\frac{z-\bar{z}}{\pi}\frac{\bar{r}^{\frac 12}}{r^{\frac 52-\frac 1q}}
F'(\xi^2)\eta(\bar{r},\bar{z})d\bar{r} d\bar{z}\bigr\|_{L^q(\Omega)}\lesssim
R^{1-\frac 3p+\frac 3q}\|\eta\|_{L^{p}}+R^{1-\frac 1p+\frac 3q}\|\eta\|_{L^{3p}}.
\eeno
Taking $R=\Bigl(\frac{\|\eta\|_{L^p}}{\|\eta\|_{L^{3p}}}\Bigr)^{\frac p2}$ in the above inequality gives rise to
\begin{equation}\label{3.5}
\bigl\|\int_{I_{2}}\frac{z-\bar{z}}{\pi}\frac{\bar{r}^{\frac 12}}{r^{\frac 52-\frac 1q}}
F'(\xi^2)\eta(\bar{r},\bar{z})d\bar{r} d\bar{z}\bigr\|_{L^q(\Omega)}
\lesssim\|\eta\|_{L^p}^{\frac{p-1}{2}+\frac{3p}{2q}}\|\eta\|_{L^{3p}}^{1-\frac{p-1}{2}-\frac{3p}{2q}}.
\end{equation}
Due to
$\bigl\|r^{\frac1q}\frac{u^r}{r}\bigr\|_{L^q(\Omega)}=\bigl\|\frac{u^r}{r}\bigr\|_{L^q},$ \eqref{3.4} together with \eqref{3.5}
ensures \eqref{3.1} for any $q\in \bigl]\frac{3p}{3-p},\infty\bigr[$.

The end-point
case when $q=\infty$ follows exactly along the same line. This completes the proof of Proposition \ref{lem2.6.1}.
\end{proof}

\subsection{The estimates of the solution operator $S(t)$}

The goal of this subsection is to present the estimates of the solution operator $S(t),$ which will be used in  Section \ref{sec3}.

\begin{prop}\label{lemsemi}
{\sl  Let $S(t)$ the solution operator given by \eqref{lemexpresion1}. Then this  family $\bigl(S(t)\bigr)_{t\geq0}$ are strongly continuous semigroups
of bounded linear operators in $L^m(\Omega)$ for any $m\in[1,\infty[$. Moreover, for
$1\leq p\leq q\leq \infty$, there holds

\begin{itemize}
\item[(1)] For any $\alpha,~\beta$ satisfying $\alpha+\beta\leq0,~a\geq-1$ and $\beta\geq-1$,
and any $f=(f^r,f^z)\in L^p(\Omega)^2$, there holds
\begin{equation}\label{semi1}
\|r^\alpha S(t)\div_* (r^{\beta}f)\|_{L^q(\Omega)}\leq\frac{C}{t^{\frac 12-\frac{\alpha+\beta}{2}+\frac 1p-\frac 1q}}\|f\|_{L^p(\Omega)}.
\end{equation}
In particular, taking $\alpha=\beta=0$, we have
\begin{equation}\label{semi2}
\|S(t)\div_* f\|_{L^q(\Omega)}\leq\frac{C}{t^{\frac 12+\frac 1p-\frac 1q}}\|f\|_{L^p(\Omega)}.
\end{equation}

\item[(2)] For any $\alpha,~\beta$ satisfying $\alpha+\beta\leq1,~\alpha\geq-1$ and $\beta\geq-1$, and any $g\in L^p(\Omega)$, there holds
\begin{equation}\label{semi3}
\|r^\alpha S(t) (r^{\beta-1} g)\|_{L^q(\Omega)}
\leq\frac{C}{t^{\frac 12-\frac{\alpha+\beta}{2}+\frac 1p-\frac 1q}}\|g\|_{L^p(\Omega)},
\end{equation}
In particular, taking $\alpha=0,~\beta=1$, and $\alpha=\beta=0$, we have
\begin{equation}\label{semi4}
\|S(t) g\|_{L^q(\Omega)}
\leq\frac{C}{t^{\frac 1p-\frac 1q}}\|g\|_{L^p(\Omega)},\quad
\|S(t) \bigl(\frac{g}{r}\bigr)\|_{L^q(\Omega)}
\leq\frac{C}{t^{\frac12+\frac 1p-\frac 1q}}\|g\|_{L^p(\Omega)}.
\end{equation}

\item[(3)] For any $\delta\in\bigl[-1,\frac12\bigr],~m\in[1,\infty[$,
and any $g$ satisfying $r^\delta g\in L^m(\Omega)$, we have
\begin{equation}\label{semicontinuous}
\|r^\delta S(t) g-r^\delta g\|_{L^m(\Omega)}\rightarrow 0,\quad \mbox{as}\ t\rightarrow 0.
\end{equation}

\end{itemize}}
\end{prop}
\begin{proof}
The boundedness of the semigroup $\bigl(S(t)\bigr)_{t\geq0}$ is shown in \eqref{semi4}.
Then in order to prove $\bigl(S(t)\bigr)_{t\geq0}$ is strongly continuous in
$L^m(\Omega)$ for any $m\in[1,\infty[$, we only need to verify the continuity at the origin,
which is a direct consequence of \eqref{semicontinuous} (with $\delta=0$). Hence it remains
to prove the estimates (\ref{semi1}-\ref{semi4}), which we handle term by term below.

(1) By integration by parts, we write
\begin{equation}\begin{split}\label{semi5}
r^\alpha\bigl(& S(t)\div_* (r^{\beta}f)\bigr)(r,z)\\
&=\frac{1}{4\pi t}
\int_\Omega \frac{\bar{r}^{\frac 12+\beta}}{r^{\frac 12-\alpha}}\exp\Bigl(-\frac{(r-\bar{r})^2+(z-\bar{z})^2}{4t}\Bigr)
\cdot\bigl(A_r f^r+A_z f^z\bigr)(\bar{r},\bar{z})\,d\bar{r}\,d\bar{z},
\end{split}\end{equation}
where
$$A_r(\bar{r},\bar{z})=\frac{t}{r\bar{r}^2}{H}'\bigl(\frac{t}{r\bar{r}}\bigr)
-\bigl(\frac{1}{2\bar{r}}+\frac{r-\bar{r}}{2t}\bigr){H}\bigl(\frac{t}{r\bar{r}}\bigr),\quad
A_z(\bar{r},\bar{z})=-\frac{z-\bar{z}}{2t}{H}\bigl(\frac{t}{r\bar{r}}\bigr).$$

\no$\bullet$ Let us first handle  the  term $\bigl|A_r+\frac{r-\bar{r}}{2t}{H}\bigl(\frac{t}{r\bar{r}}\bigr)\bigr|$ .

If $(\alpha,\beta)\in\Omega_1\eqdefa\{(\alpha,\beta)|0\leq\frac12-\frac{\alpha+\beta}{2}\leq\frac32
,~0\leq\frac12-\beta\leq\frac32,~\beta\leq\alpha\}$, we can divide the integral area into
$\{\bar{r}\geq \frac r2\}$ and $\{\bar{r}< \frac r2\}$.
When $\bar{r}\geq \frac r2$, we can deduce from Corollary \ref{corH} that
\begin{equation}\begin{split}\label{semi6}
&\frac{\bar{r}^{\frac 12+\beta}}{r^{\frac 12-\alpha}}
\exp\Bigl(-\frac{(r-\bar{r})^2+(z-\bar{z})^2}{4t}\Bigr)
\bigl|A_r+\frac{r-\bar{r}}{2t}{H}\bigl(\frac{t}{r\bar{r}}\bigr)\bigr|\\
&\lesssim \Bigl(\frac{t}{r^{\frac 32-\alpha}\bar{r}^{\frac 32-\beta}}{H}'\bigl(\frac{t}{r\bar{r}}\bigr)
+\frac{1}{r^{\frac 12-\alpha}\bar{r}^{\frac 12-\beta}}{H}\bigl(\frac{t}{r\bar{r}}\bigr)\Bigr)\cdot
\exp\Bigl(-\frac{(r-\bar{r})^2+(z-\bar{z})^2}{4t}\Bigr)\\
&\lesssim \Bigl(\frac{t}{r^{\frac 32-\alpha}\bar{r}^{\frac 32-\beta}}\bigl|\frac{r\bar{r}}{t}\bigr|^{\frac32-\frac{\alpha+\beta}{2}}
+\frac{1}{r^{\frac 12-\alpha}\bar{r}^{\frac 12-\beta}}\bigl|\frac{r\bar{r}}{t}\bigr|^{\frac12-\frac{\alpha+\beta}{2}}\Bigr)
\cdot\exp\Bigl(-\frac{(r-\bar{r})^2+(z-\bar{z})^2}{4t}\Bigr)\\
&\lesssim \frac{1}{t^{\frac12-\frac{\alpha+\beta}{2}}}\cdot\exp\Bigl(-\frac{(r-\bar{r})^2+(z-\bar{z})^2}{5t}\Bigr).
\end{split}\end{equation}
And when $\bar{r}<\frac r2$, there then holds ${r}<2|\bar{r}-r|$,
another use of Corollary \ref{corH} gives
\begin{equation}\begin{split}\label{semi7}
&\frac{\bar{r}^{\frac 12+\beta}}{r^{\frac 12-\alpha}}
\exp\Bigl(-\frac{(r-\bar{r})^2+(z-\bar{z})^2}{4t}\Bigr)
\bigl|A_r+\frac{r-\bar{r}}{2t}{H}\bigl(\frac{t}{r\bar{r}}\bigr)\bigr|\\
&\lesssim \Bigl(\frac{t}{r^{\frac 32-\alpha}\bar{r}^{\frac 32-\beta}}\bigl|\frac{r\bar{r}}{t}\bigr|^{\frac32-\beta}
+\frac{1}{r^{\frac 12-\alpha}\bar{r}^{\frac 12-\beta}}\bigl|\frac{r\bar{r}}{t}\bigr|^{\frac12-\beta}\Bigr)
\exp\Bigl(-\frac{(r-\bar{r})^2+(z-\bar{z})^2}{4t}\Bigr)\\
&\lesssim \frac{{r}^{\alpha-\beta}}{t^{\frac12-\beta}}\cdot
\bigl(\frac{5t}{(r-\bar{r})^2+(z-\bar{z})^2}\bigr)^{\frac{\alpha-\beta}2}\cdot\exp\Bigl(-\frac{(r-\bar{r})^2+(z-\bar{z})^2}{5t}\Bigr)\\
&\lesssim \frac{1}{t^{\frac12-\frac{\alpha+\beta}{2}}}\cdot\exp\Bigl(-\frac{(r-\bar{r})^2+(z-\bar{z})^2}{5t}\Bigr).
\end{split}\end{equation}

If $(\alpha,\beta)\in\Omega_2\eqdefa\{(\alpha,\beta)|0\leq\frac12-\frac{\alpha+\beta}{2}\leq\frac32
,~0\leq\frac12-\alpha\leq\frac32,~\alpha\leq \beta\}$, we divide the integral area in a different way as
$\{\bar{r}\leq 2r\}$ and $\{\bar{r}>2r\}$. Similar to the previous estimates,
when $\bar{r}\leq 2r$, we have
\begin{equation}\begin{split}\label{semi8}
&\frac{\bar{r}^{\frac 12+\beta}}{r^{\frac 12-\alpha}}
\exp\Bigl(-\frac{(r-\bar{r})^2+(z-\bar{z})^2}{4t}\Bigr)
\bigl|A_r+\frac{r-\bar{r}}{2t}{H}\bigl(\frac{t}{r\bar{r}}\bigr)\bigr|\\
&\lesssim \Bigl(\frac{t}{r^{\frac 32-\alpha}\bar{r}^{\frac 32-\beta}}\bigl|\frac{r\bar{r}}{t}\bigr|^{\frac32-\frac{\alpha+\beta}{2}}
+\frac{1}{r^{\frac 12-\alpha}\bar{r}^{\frac 12-\beta}}\bigl|\frac{r\bar{r}}{t}\bigr|^{\frac12-\frac{\alpha+\beta}{2}}\Bigr)
\cdot\exp\Bigl(-\frac{(r-\bar{r})^2+(z-\bar{z})^2}{4t}\Bigr)\\
&\lesssim \frac{1}{t^{\frac12-\frac{\alpha+\beta}{2}}}\cdot\exp\Bigl(-\frac{(r-\bar{r})^2+(z-\bar{z})^2}{5t}\Bigr).
\end{split}\end{equation}
And when $\bar{r}> 2r$, there then holds $\bar{r}<2|\bar{r}-r|$, thus we deduce
\begin{equation}\begin{split}\label{semi9}
&\frac{\bar{r}^{\frac 12+\beta}}{r^{\frac 12-\alpha}}
\exp\Bigl(-\frac{(r-\bar{r})^2+(z-\bar{z})^2}{4t}\Bigr)
\bigl|A_r+\frac{r-\bar{r}}{2t}{H}\bigl(\frac{t}{r\bar{r}}\bigr)\bigr|\\
&\lesssim \Bigl(\frac{t}{r^{\frac 32-\alpha}\bar{r}^{\frac 32-\beta}}\bigl|\frac{r\bar{r}}{t}\bigr|^{\frac32-\alpha}
+\frac{1}{r^{\frac 12-\alpha}\bar{r}^{\frac 12-\beta}}\bigl|\frac{r\bar{r}}{t}\bigr|^{\frac12-\alpha}\Bigr)
\exp\Bigl(-\frac{(r-\bar{r})^2+(z-\bar{z})^2}{4t}\Bigr)\\
&\lesssim \frac{\bar{r}^{\beta-\alpha}}{t^{\frac12-\alpha}}\cdot
\bigl(\frac{5t}{(r-\bar{r})^2+(z-\bar{z})^2}\bigr)^{\frac{\beta-\alpha}2}\cdot\exp\Bigl(-\frac{(r-\bar{r})^2+(z-\bar{z})^2}{5t}\Bigr)\\
&\lesssim \frac{1}{t^{\frac12-\frac{\alpha+\beta}{2}}}\cdot\exp\Bigl(-\frac{(r-\bar{r})^2+(z-\bar{z})^2}{5t}\Bigr).
\end{split}\end{equation}

Thus combining the estimates (\ref{semi1}-\ref{semi4}), we conclude that whenever $(\alpha,\beta)\in\Omega_1\bigcup\Omega_2$,
i.e. $\alpha,~\beta$ satisfy $\alpha+\beta\leq 1$,~$\alpha\geq-1$ and $\beta\geq-1$, there holds
\begin{equation}\begin{split}\label{semi10}
\frac{\bar{r}^{\frac 12+\beta}}{r^{\frac 12-\alpha}}
&\exp\Bigl(-\frac{(r-\bar{r})^2+(z-\bar{z})^2}{4t}\Bigr)
\cdot\bigl|A_r+\frac{r-\bar{r}}{2t}{H}\bigl(\frac{t}{r\bar{r}}\bigr)\bigr|\\
&\qquad\qquad\qquad\qquad\lesssim \frac{1}{t^{\frac12-\frac{\alpha+\beta}{2}}}\cdot\exp\Bigl(-\frac{(r-\bar{r})^2+(z-\bar{z})^2}{5t}\Bigr).
\end{split}\end{equation}

\no$\bullet$ For $|A_z|$ term in the integrand \eqref{semi5}.
When $\bar{r}>2r$ or $\bar{r}\leq \frac r2$, there then holds $\bar{r}+r<3|\bar{r}-r|$.
If in addition $\alpha+\beta\geq-4$,~$\alpha\geq-1$ and $\beta\geq-2$, there then exists a positive constant $\gamma$ so that
$\max\bigl\{0,\frac12-\alpha,-\frac12-\beta,-\frac{1+\alpha+\beta}{2}\bigr\}\leq \gamma\leq\frac32$.
Then we deduce from Corollary \ref{corH} that
\begin{align*}
&\frac{\bar{r}^{\frac 12+\beta}}{r^{\frac 12-\alpha}}
\exp\Bigl(-\frac{(r-\bar{r})^2+(z-\bar{z})^2}{4t}\Bigr)
\Bigl(|A_z|+\bigl|\frac{r-\bar{r}}{2t}{H}\bigl(\frac{t}{r\bar{r}}\bigr)\bigr|\Bigr)\\
&\lesssim \frac{\bar{r}^{\frac 12+\beta}}{r^{\frac 12-\alpha}}\frac{|r-\bar{r}|+|z-\bar{z}|}{t}\bigl|\frac{r\bar{r}}{t}\bigr|^{\gamma}
\bigl(\frac{5t}{(r-\bar{r})^2+(z-\bar{z})^2}\bigr)^{\frac{1+2\gamma+\alpha+\beta}2}\exp\Bigl(-\frac{(r-\bar{r})^2+(z-\bar{z})^2}{5t}\Bigr)\\
&\lesssim \frac{1}{t^{\frac12-\frac{\alpha+\beta}{2}}}\exp\Bigl(-\frac{(r-\bar{r})^2+(z-\bar{z})^2}{5t}\Bigr).
\end{align*}
And when $\frac r2\leq \bar{r}\leq 2r$, if in addition, $-3\leq\alpha+\beta\leq 0$, we have
\begin{align*}
&\frac{\bar{r}^{\frac 12+\beta}}{r^{\frac 12-\alpha}}
\exp\Bigl(-\frac{(r-\bar{r})^2+(z-\bar{z})^2}{4t}\Bigr)
\Bigl(|A_z|+\bigl|\frac{r-\bar{r}}{2t}{H}\bigl(\frac{t}{r\bar{r}}\bigr)\bigr|\Bigr)\\
&\lesssim \frac{\bar{r}^{\frac 12+\beta}}{r^{\frac 12-\alpha}}
\frac{|r-\bar{r}|+|z-\bar{z}|}{t}\bigl|\frac{r\bar{r}}{t}\bigr|^{-\frac{\alpha+\beta}{2}}
\bigl(\frac{5t}{(r-\bar{r})^2+(z-\bar{z})^2}\bigr)^{\frac12}\exp\Bigl(-\frac{(r-\bar{r})^2+(z-\bar{z})^2}{5t}\Bigr)\\
&\lesssim \frac{1}{t^{\frac12-\frac{\alpha+\beta}{2}}}\cdot\exp\Bigl(-\frac{(r-\bar{r})^2+(z-\bar{z})^2}{5t}\Bigr).
\end{align*}

Therefore as long as $\alpha+\beta\leq 0$,~$\alpha\geq-1$ and $\beta\geq-2$, there holds
\begin{equation}\begin{split}\label{semi11}
\frac{\bar{r}^{\frac 12+\beta}}{r^{\frac 12-\alpha}}&
\exp\Bigl(-\frac{(r-\bar{r})^2+(z-\bar{z})^2}{4t}\Bigr)
\Bigl(|A_z|+\bigl|\frac{r-\bar{r}}{2t}{H}\bigl(\frac{t}{r\bar{r}}\bigr)\bigr|\Bigr)\\
&\qquad\qquad\qquad\qquad\lesssim \frac{1}{t^{\frac12-\frac{\alpha+\beta}{2}}}\cdot\exp\Bigl(-\frac{(r-\bar{r})^2+(z-\bar{z})^2}{5t}\Bigr).
\end{split}\end{equation}

By combining \eqref{semi10} with \eqref{semi11}, we achieve
\begin{equation*}\label{semi12}
\frac{\bar{r}^{\frac 12+\beta}}{r^{\frac 12-\alpha}}
\exp\Bigl(-\frac{(r-\bar{r})^2+(z-\bar{z})^2}{4t}\Bigr)
\bigl(|A_r|+|A_z|\bigr)\lesssim \frac{1}{t^{\frac12-\frac{\alpha+\beta}{2}}}\cdot\exp\Bigl(-\frac{(r-\bar{r})^2+(z-\bar{z})^2}{5t}\Bigr),
\end{equation*}
provided $\alpha+\beta\leq0,~a\geq-1$ and $\beta\geq-1$.
And then \eqref{semi1} follows from  \eqref{semi5} and Young's inequality in two space dimension.

(2) It follows from the proof of \eqref{semi10} that
\begin{equation*}
\frac{1}{r^{\frac 12-\alpha}\bar{r}^{\frac 12-\beta}}
\bigl|{H}\bigl(\frac{t}{r\bar{r}}\bigr)\bigr|
\exp\Bigl(-\frac{(r-\bar{r})^2+(z-\bar{z})^2}{4t}\Bigr)
\lesssim \frac{1}{t^{\frac12-\frac{\alpha+\beta}{2}}}\exp\Bigl(-\frac{(r-\bar{r})^2+(z-\bar{z})^2}{5t}\Bigr),
\end{equation*}
whenever $\alpha+\beta\leq1,~\alpha\geq-1$ and $\beta\geq-1$.
Then by virtue of \eqref{lemexpresion1},  we get, by applying Young's inequality, that there holds \eqref{semi3}.

(3) In view of \eqref{lemexpresion1},
we get, by using  changes of variables that
\begin{equation}\label{semi13}
\bigl(r^\delta S(t) g-r^\delta g\bigr)(r,z)=\frac{1}{4\pi}\int_{\Omega}
\Psi(r,z,\rho,\xi,t)\cdot\exp\Bigl(-\frac{\rho^2+\xi^2}{4}\Bigr)\, d\rho d\xi,
\end{equation}
for all $(r,z)\in\Omega$, and where
\begin{equation}\label{semi14}
\Psi(r,z,\rho,\xi,t)\eqdefa r^\delta\bigl(\frac{r+\sqrt{t}\rho}{r}\bigr)^{\frac12}
{H}\Bigl(\frac{t}{r(r+\sqrt{t}\rho)}\Bigr) g(r+\sqrt{t}\rho,z+\sqrt{t}\xi)-r^\delta g(r,z).
\end{equation}
Notice that $\frac12-\delta\in[0,\frac32]$, applying Corollary \ref{corH} gives
\beno
\begin{split}
&\bigl(\frac{\bar{r}}{r}\bigr)^{\frac12-\delta}
{H}\Bigl(\frac{t}{r\bar{r}}\Bigr)\leq C\quad\mbox{if}\quad \bar{r}\leq2r \andf\\
&\bigl(\frac{\bar{r}}{r}\bigr)^{\frac12-\delta}
{H}\Bigl(\frac{t}{r\bar{r}}\Bigr)\leq C\bigl(\frac{\bar{r}}{r}\bigr)^{\frac12-\delta}
\bigl(\frac{r\bar{r}}{t}\bigr)^{\frac12-\delta}\leq C\bigl(\frac{|\bar{r}-r|^2}{t}\bigr)^{\frac12-\delta}\quad\mbox{if}\quad \bar{r}>2r,
\end{split}
\eeno
which implies for any given  $(\rho,\xi, t)$, we have
\begin{equation}\label{semi15}
\Psi(r,z,\rho,\xi,t)\lesssim (1+\rho^{1-2\delta})\cdot
(r+\sqrt{t}\rho)^\delta g(r+\sqrt{t}\rho,z+\sqrt{t}\xi)-r^\delta g(r,z)\in L^m(\Omega).
\end{equation}
Moreover, noting that $H(t)=1+\mathcal{O}(t),$ as $t\rightarrow0$, it is easy to observe that
$\Psi(r,z,\rho,\xi,t)\rightarrow 0$ as $t\rightarrow 0$.
Then  Lebesgue dominated convergence theorem ensures that
$$\|\Psi(\cdot,\cdot,\rho,\xi,t)\|_{L^m(\Omega)}\rightarrow 0,\ \mbox{as}\ t\rightarrow 0,$$
from which and \eqref{semi13}, another use of Lebesgue's dominated convergence theorem gives
$$
\|r^\delta S(t) g-r^\delta g\|_{L^m(\Omega)}\leq\frac{1}{4\pi}\int_{\Omega}
\|\Psi(\cdot,\cdot,\rho,\xi,t)\|_{L^m(\Omega)}\exp\Bigl(-\frac{\rho^2+\xi^2}{4}\Bigr)\, d\rho d\xi
\rightarrow 0,\ \mbox{as}\ t\rightarrow 0.
$$
This completes the proof of the proposition.
\end{proof}

\section{Local existence of solutions to \eqref{1.1} in critical spaces}\label{sec3}

The purpose of this section is to investigate the local existence and uniqueness of the mild solutions to
\eqref{otut} in the spirit of \cite{GS, Kato}. In view of \eqref{lemexpresion1},
we  rewrite  the systems \eqref{otut} as
\begin{equation}\label{inteeqt}
\left\{
\begin{array}{l}
\displaystyle \ot(t)=S(t)\ot_0-\int_0^t S(t-s)\Bigl(\div_*\bigl(\widetilde{u}(s)\ot(s)\bigr)-\frac{\pa_z|\ut(s)|^2}{r}\Bigr)\,ds,\quad t>0, \\
\displaystyle \ut(t)=S(t)\ut_0-\int_0^t S(t-s)\Bigl(\div_*\bigl(\widetilde{u}(s)\ut(s)\bigr)-\frac{2\ut(s)\ur(s)}{r}\Bigr)ds,\quad t>0,
\end{array}
\right.
\end{equation}

 Now we  present the proof of Theorem \ref{thmlocal}.

\begin{proof}[Proof of Theorem \ref{thmlocal}] The main idea to prove Theorem \ref{thmlocal} is to apply fixed point argument for the integral formulation \eqref{inteeqt}.
Toward this, for any  $T>0$, we introduce the functional space
\begin{equation}\begin{split}\label{fix1}
X_T\eqdefa\bigl\{(\ot,\ut)\in C\bigl(]0,T];L^{\frac 43}&(\Omega)\bigr)\times C\bigl(]0,T];L^{4}(\Omega)
\bigr)\ \big|\\
&r^{-\frac{3}{10}}\ut\in C\bigl(]0,T];L^{2}(\Omega)\bigr) \andf \|(\omega^\theta,\ut)\|_{X_T}<\infty\bigr\},
\end{split}\end{equation}
where
\begin{equation}\label{fix2}
\|(\ot,\ut)\|_{X_T}\eqdefa\sup_{0<t\leq T}\bigl(t^{\frac 14}\|\ot(t)\|_{L^{\frac 43}(\Omega)}
+t^{\frac 14}\|\ut(t)\|_{L^{4}(\Omega)}+t^{\frac 3{20}}\|r^{-\frac{3}{10}}\ut(t)\|_{L^{2}(\Omega)}\bigr).
\end{equation}
For convenience, sometimes we may
abuse the notation $\|\ot\|_{X_T}=\sup\limits_{0<t\leq T}t^{\frac 14}\|\ot\|_{L^{\frac 43}(\Omega)}$,
~$\|\ut\|_{X_T}=\sup\limits_{0<t\leq T}\bigl(t^{\frac 14}\|\ut(t)\|_{L^{4}(\Omega)}+t^{\frac 3{20}}\|r^{-\frac{3}{10}}\ut(t)\|_{L^{2}(\Omega)}\bigr)$,
and $\ot\in X_T$ (resp. $\ut\in X_T$) means that $\|\ot\|_{X_T}<\infty$ (resp. $\|\ut\|_{X_T}<\infty$).

\no$\bullet$ \underline{The estimate of $\ot$ term}

In view of \eqref{semi4}, $S(t)\ot_0\in X_T$ for any $T>0$,
and there exists a universal constant $C_1>0$ such that
for any $T>0$, we have
\beq\label{fix3}
\sup_{0<t\leq T}t^{\frac 14}\|S(t)\ot_0\|_{L^{\frac 43}(\Omega)}
\leq C_1\|\ot_0\|_{L^1(\Omega)}.
\eeq

On the other hand,
since $L^1(\Omega)\bigcap L^{\frac 43}(\Omega)$
is dense in $L^1(\Omega).$ For any $\e>0,$  there exists $\widetilde{\omega}^\theta_0\in L^1(\Omega)\bigcap L^{\frac 43}(\Omega)$
satisfying $\|\widetilde{\omega}^\theta_0-\ot_0\|_{L^1(\Omega)}<\varepsilon.$
Then it follows from \eqref{semi4} that
\beno
\begin{split}
t^{\f14}\|S(t)\ot_0\|_{L^{\f43}(\Omega)}\leq & t^{\f14}\|S(t)(\ot_0-\widetilde{\omega}^\theta_0)\|_{L^{\f43}(\Omega)}+t^{\f14}\|S(t)\widetilde{\omega}^\theta_0\|_{L^{\f43}(\Omega)}\\
\leq & C\bigl(\|\widetilde{\omega}^\theta_0-\ot_0\|_{L^1(\Omega)}+t^{\f14}\|\widetilde{\omega}^\theta_0\|_{L^{\f43}(\Omega)}\bigr)\\
\leq& C\bigl(\e+t^{\f14}\|\widetilde{\omega}^\theta_0\|_{L^{\f43}(\Omega)}\bigr),
\end{split}
\eeno
which implies
\beq \label{fix4}
\lim_{t\to 0}t^{\f14}\|S(t)\ot_0\|_{L^{\f43}(\Omega)}=0.
\eeq
Let us denote $\wt{u}(\ot)(t)$ be velocity field determined by the vorticity $\ot e_\theta$
via the axisymmetric Biot-Savart law \eqref{2.2.1}.
Given $(\ot_1,\ut_1),~(\ot_2,\ut_2)\in X_T$,  for any $t\in [0,T],$  we define the mapping $\mathcal{F}^\omega$ on $X_T\times X_T$ by
\begin{equation}\label{fix5}
\mathcal{F}^\omega\bigl((\ot_1,\ut_1),(\ot_2,\ut_2)\bigr)(t)\eqdefa\int_0^tS(t-s)\cdot\bigl(\div_*(\widetilde{u}(\ot_1)(s)\ot_2(s))
-\frac{\pa_z(\ut_1(s)\ut_2(s))}{r}\bigr)ds.
\end{equation}
Then for any $p\in [1,\frac{10}{7}[,$ we deduce from  \eqref{semi1} (with $~\alpha=0,~\beta=-\frac25$),~\eqref{semi2} and \eqref{2.2.5} that
 for any $t\in]0,T]$,
\begin{equation}\label{fix6}\begin{split}
&t^{1-\frac 1p}\|\mathcal{F}^\omega\bigl((\omega_1,\ut_1),(\omega_2,\ut_2)\bigr)(t)\|_{L^p(\Omega)}\\
&\leq C
t^{1-\frac 1p}\int_0^t\frac{\|\widetilde{u}(\ot_1)(s)\cdot\ot_2(s)\|_{L^1(\Omega)} }{(t-s)^{\frac 32-\frac 1p}}
+\frac{\|r^{-\frac{3}{5}}\ut_1(s)\ut_2(s)\|_{L^1(\Omega)} }{(t-s)^{\frac 32-\frac 1p+\frac15}}\,ds\\
&\leq C\, t^{1-\frac 1p}\int_0^t\frac{\|\widetilde{u}(\ot_1)(s)\|_{L^4(\Omega)}\|\ot_2(s)\|_{L^{\frac 43}(\Omega)}}{(t-s)^{\frac 32-\frac 1p}}
+\frac{\|r^{-\frac{3}{10}}\ut_1(s)\|_{L^2(\Omega)}\|r^{-\frac{3}{10}}\ut_2(s)\|_{L^{2}(\Omega)}}{(t-s)^{\frac 32-\frac 1p+\frac15}}\,ds\\
&\leq C\, t^{1-\frac 1p}\int_0^t \frac{\|\ot_1\|_{X_T}\|\ot_2\|_{X_T}}{(t-s)^{\frac 32-\frac 1p}\cdot s^{\frac 12}}
+\frac{\|\ut_1\|_{X_T}\|\ut_2\|_{X_T}}{(t-s)^{\frac {17}{10}-\frac 1p}\cdot s^{\frac {3}{10}}}\,ds\\
&\leq A^\omega_p\bigl(\|\ot_1\|_{X_T}\|\ot_2\|_{X_T}+\|\ut_1\|_{X_T}\|\ut_2\|_{X_T}\bigr).
\end{split}
\end{equation}

\no$\bullet$ \underline{The estimate of $\ut$ terms}

Thanks to \eqref{semi3} (with $\alpha=-\frac{3}{10},~\beta=\frac{13}{10}$),
by a similar derivation of \eqref{fix3},~\eqref{fix4}, we get
\begin{equation}\begin{split}\label{fix7}
&\sup_{0<t\leq T}\bigl(t^{\frac 14}\|S(t)\ut_0\|_{L^{4}(\Omega)}
+t^{\frac{3}{20}}\|r^{-\frac{3}{10}}S(t)\ut_0\|_{L^{2}(\Omega)}\bigr)\leq C_1\bigl(\|\ut_0\|_{L^{2}(\Omega)}
+\|r^{-\frac{3}{10}}\ut_0\|_{L^{\frac{20}{13}}(\Omega)}\bigr),\\
&\andf\lim_{t\to 0}\bigl(t^{\frac 14}\|S(t)\ut_0\|_{L^{4}(\Omega)}+t^{\frac{3}{20}}\|r^{-\frac{3}{10}}S(t)\ut_0\|_{L^{2}(\Omega)}\bigr)=0,
\end{split}\end{equation}
Given $(\ot_1,\ut_1),~(\ot_2,\ut_2)\in X_T,$ for any $t\in [0,T],$ we define the mapping $\mathcal{F}^u$ on $X_T\times X_T$ by
\begin{equation}\label{fix8}
\mathcal{F}^u\bigl((\omega_1,\ut_1),(\omega_2,\ut_2)\bigr)(t)\eqdefa\int_0^tS(t-s)\Bigl(\div_*\bigl(\widetilde{u}(\ot_1)(s)\cdot\ut_2(s)\bigr)
-\frac{2\ur(\ot_1)(s)\cdot\ut_2(s)}{r}\Bigr)\,ds.
\end{equation}
Then for any $q\in[2,\infty[$, we deduce from  \eqref{semi2}, \eqref{semi4} and \eqref{2.2.5} that
for  any $t\in]0,T]$
\begin{equation}\label{fix9}\begin{split}
&t^{\frac 12-\frac 1q}\|\mathcal{F}^u\bigl((\omega_1,\ut_1),(\omega_2,\ut_2)\bigr)(t)\|_{L^q(\Omega)}\\
&\leq
t^{\frac 12-\frac 1q}\int_0^t\frac{C}{(t-s)^{1-\frac 1q}}\|\widetilde{u}(\ot_1)(s)\cdot\ut_2(s)\|_{L^2(\Omega)}\,ds\\
&\leq t^{\frac 12-\frac 1q}\int_0^t\frac{C}{(t-s)^{1-\frac 1q}}\|\widetilde{u}(\ot_1)(s)\|_{L^4(\Omega)}\|\ut_2(s)\|_{L^4(\Omega)}\,ds\\
&\leq t^{\frac 12-\frac 1q}\int_0^t\frac{C}{(t-s)^{1-\frac 1q}}\|\ot_1(s)\|_{L^{\frac 43}(\Omega)}\|\ut_2(s)\|_{L^4(\Omega)}\,ds\\
&\leq t^{\frac 12-\frac 1q}\int_0^t\frac{C}{(t-s)^{1-\frac 1q}}\frac{\|\ot_1\|_{X_t}\|\ut_2\|_{X_t}}{s^{\frac 12}}\,ds\leq A^u_q\|\ot_1\|_{X_T}\|\ut_2\|_{X_T}.
\end{split}
\end{equation}

Along the same line, for  any $\kappa\in[\frac{20}{13},4[$, we deduce from  \eqref{semi1},~\eqref{semi3} (with $\alpha=-\frac{3}{10},~\beta=\frac{3}{10}$)
and \eqref{2.2.5} that
for  any $t\in]0,T]$
\begin{equation}\label{fix10}\begin{split}
&t^{\frac {13}{20}-\frac 1\kappa}\|r^{-\frac{3}{10}}\mathcal{F}^u\bigl((\omega_1,\ut_1),(\omega_2,\ut_2)\bigr)(t)\|_{L^\kappa(\Omega)}\\
&\leq
t^{\frac {13}{20}-\frac 1\kappa}\int_0^t\frac{C}{(t-s)^{\frac54-\frac 1\kappa}}\|\widetilde{u}(\ot_1)(s)\cdot r^{-\frac{3}{10}}\ut_2(s)\|_{L^{\f43}(\Omega)}\,ds\\
&\leq t^{\frac {13}{20}-\frac 1\kappa}\int_0^t\frac{C}{(t-s)^{\frac54-\frac 1\kappa}}
\|\widetilde{u}(\ot_1)(s)\|_{L^4(\Omega)}\|r^{-\frac{3}{10}}\ut_2(s)\|_{L^2(\Omega)}\,ds\\
&\leq t^{\frac {13}{20}-\frac 1\kappa}\int_0^t\frac{C}{(t-s)^{\frac54-\frac 1\kappa}}\frac{\|\ot_1\|_{X_t}\|\ut_2\|_{X_t}}{s^{\frac 25}}\,ds\leq B^u_k\|\ot_1\|_{X_T}\|\ut_2\|_{X_T}.
\end{split}
\end{equation}

\no$\bullet$ \underline{Fixed point argument}

For $(\ot_1,\ut_1),~(\ot_2,\ut_2)\in X_T$, we consider the following bilinear map
\begin{equation}\label{bilinear}
\mathcal{F}\bigl((\omega_1,\ut_1),(\omega_2,\ut_2)\bigr)\eqdefa
\bigl(\cF^\omega,\cF^u\bigr)\bigl((\omega_1,\ut_1),(\omega_2,\ut_2)\bigr),
\end{equation}
with $\cF^\omega,~\cF^u$ given by \eqref{fix5} and \eqref{fix8} respectively.

By virtue of  \eqref{fix4} and \eqref{fix7}, for any $\ot_0
\in L^1(\Omega),~\ut_0\in L^2(\Omega)$ with $r^{-\frac{3}{10}}\ut_0\in L^{\frac{20}{13}}(\Omega)$, there exists a positive time $T$ such that
\begin{equation}\begin{split}
EL(T)&\eqdefa\sup_{t\in [0,T]}\bigl(t^{\f14}\bigl\|S(t)\ot_0\bigr\|_{L^{\f43}(\Omega)}
+t^{\f14}\|S(t)\ut_0\|_{L^4(\Omega)}+t^{\frac{3}{20}}\|r^{-\frac{3}{10}}S(t)\ut_0\|_{L^{2}(\Omega)}\bigr)\\
& \leq \f1{4\bigl(A^\omega_{4/3}+A^u_4+B^u_2\bigr)},
\end{split}\end{equation}
where the constants $A^\omega_{4/3},~A^u_4$ and $B^u_2$ are determined by \eqref{fix6},~\eqref{fix9} and \eqref{fix10} respectively. Then we deduce from
Lemma \ref{lem4.5}
that \eqref{inteeqt} has a unique solution $(\ot, u^\theta)$ in $X_T$. Furthermore, by virtue of \eqref{fix4} and \eqref{fix7},
for any $\e>0,$ there exists $T_\e>0$ so that
\beno
EL(T_\e)\leq \e.
\eeno
Then Lemma \ref{lem4.5} ensures that
\beno
\bigl\|\bigl(\ot,\ut\bigr)\bigr\|_{X_{T_\e}}\leq 2\e, \eeno
which implies that
\beq\label{fix11}
\lim_{T\to 0} \bigl\|\bigl(\ot,
\ut\bigr)\bigr\|_{X_{T}}=0.
\eeq

Futhermore, for any $T>0$, it follows from \eqref{fix3} and \eqref{fix7} that
$$EL(T)\leq C_1\bigl(\|\ot_0\|_{L^1(\Omega)}+\|\ut_0\|_{L^2(\Omega)}
+\|r^{-\frac{3}{10}}\ut_0\|_{L^{\frac{20}{13}}(\Omega)}\bigr)\leq \f1{4\bigl(A^\omega_{4/3}+A^u_4+B^u_2\bigr)},
$$
provided that $c$ in \eqref{smallc} satisfying $c\leq \f1{4C_1\bigl(A^\omega_{4/3}+A^u_4+B^u_2\bigr)}.$  This together with Lemma \ref{lem4.5}
shows that \eqref{inteeqt} has a unique global solution in $X_\infty.$

\no$\bullet$ \underline{Behavior near $t=0$}

Let us now turn to the estimate \eqref{thmlocal3}.
Let $(\ot, \ut)$ be the unique solution of \eqref{inteeqt} on $[0,T]$ obtained in the previous one step,
we denote
$$J_{p,q,\kappa}(T)\eqdefa\sup_{0\leq t\leq T}\bigl(t^{1-\frac 1{p}}\bigl\|S(t)\ot_0\bigr\|_{L^p(\Omega)}
+t^{\frac12-\frac 1q}\|S(t)\ut_0\|_{L^q(\Omega)}+t^{\frac{13}{20}-\frac 1{\kappa}}\|r^{-\frac{3}{10}}S(t)\ut_0\|_{L^{\kappa}(\Omega)}.$$
Along the same line to the proof of \eqref{fix4}, it is easy to observe that
\begin{equation}\label{fix12}
\lim\limits_{T\rightarrow0}J_{p,q,\kappa}(T)=0,\quad\forall
p\in]1,\infty],~q\in]2,\infty],~\kappa\in]{20}/{13},\infty].
\end{equation}
While it follows from the proof of \eqref{fix6}, \eqref{fix9} and \eqref{fix10} that
\beno
\begin{split}
\bigl\|\omega^\theta(t)-S(t)\omega^\theta_0\bigr\|_{L^1(\Omega)}\leq& A^\omega_p\bigl(\|\omega^\theta\|^2_{X_t}
+\|\ut\|^2_{X_t}\bigr),\\
\|\ut(t)-S(t)\ut_0\|_{L^2(\Omega)}\leq & A^u_2\|\omega^\theta\|_{X_t}\|\ut\|_{X_t},\\
\|r^{-\frac{3}{10}}\ut(t)-r^{-\frac{3}{10}}S(t)\ut_0\|_{L^{\frac{20}{13}}(\Omega)}\leq& B^u_{\frac{20}{13}}\|\omega^\theta\|_{X_t}\|\ut\|_{X_t},
\end{split}
\eeno
from which,~\eqref{semicontinuous} and \eqref{fix11}, we infer
\beq \label{fix13}
\lim_{t\to 0}\Bigl(\bigl\|\omega^\theta(t)-\omega^\theta_0\bigr\|_{L^1(\Omega)}+
\|\ut(t)-\ut_0\|_{L^2(\Omega)}+\|r^{-\frac{3}{10}}\ut(t)-r^{-\frac{3}{10}}\ut_0\|_{L^{\frac{20}{13}}(\Omega)}\Bigr)=0.
\eeq
Whereas for any $t_0>0,$ we get, by using a similar derivation of \eqref{fix6}, that
\beno
\begin{split}
\Bigl\|&\int_{t_0}^t S(t-s)\Bigl(\div_*\bigl(\widetilde{u}(s)\ot(s)\bigr)+\frac{2\ut(s)\orr(s)}{r}\Bigr)\,ds\Bigr\|_{L^1(\Omega)}\\
&\leq C\int_{t_0}^t \frac{\|\ot_1\|_{X_T}\|\ot_2\|_{X_T}}{(t-s)^{\frac 12}\cdot s^{\frac 12}}
+\frac{\|\ut_1\|_{X_T}\|\ut_2\|_{X_T}}{(t-s)^{\frac {7}{10}}\cdot s^{\frac {3}{10}}}\,ds\\
&\leq C\bigl(t_0^{-\f12}(t-t_0)^{\f12}+t_0^{-\frac {3}{10}}(t-t_0)^{\frac {3}{10}}\bigr)
\bigl(\|\ot\|_{X_t}^2+\|\ut\|_{X_t}\|\omega^r\|_{X_t}\bigr).
\end{split}
\eeno
Hence by virtue of the expression \eqref{inteeqt}, we deduce that
\beq \label{fix14}
\lim_{t\to t_0}\|\omega^\theta(t)-\omega^\theta(t_0)\|_{L^1(\Omega)}=0.
\eeq
Exactly along the same line, we can prove that
\beno
\lim_{t\to t_0}\Bigl(\|\ut(t)-\ut(t_0)\|_{L^2(\Omega)}
+\|r^{-\frac{3}{10}}\ut(t)-r^{-\frac{3}{10}}\ut(t_0)\|_{L^{\frac{20}{13}}(\Omega)}\Bigr)=0,
\eeno
This together with \eqref{fix13} and \eqref{fix14} ensures that
\begin{equation}\label{fix15}
\omega\in C([0,T];L^1(\Omega)),~\ut\in  C([0,T];L^2(\Omega)) \mbox{ and }
r^{-\frac{3}{10}}\ut\in  C([0,T];L^{\frac{20}{13}}(\Omega)).
\end{equation}

For any $q\in]2,\infty]$, by using \eqref{fix9},~\eqref{fix11}
and \eqref{fix12}, we deduce that
$M_q(T)$ are bounded, and $M_q(T)\rightarrow0$ as $T\rightarrow0$.

For the estimate of $N_\kappa(T)$ and $L_p(T)$, we shall use a bootstrap argument.
Indeed to estimate $N_\kappa(T)$, we get, by a similar derivation of \eqref{fix10}, that
\begin{equation}\label{fix17}\begin{split}
&t^{\frac {13}{20}-\frac 1\kappa}\|r^{-\frac{3}{10}}\mathcal{F}^u\bigl((\omega,\ut),(\omega,\ut)\bigr)(t)\|_{L^\kappa(\Omega)}\\
&\leq
t^{\frac {13}{20}-\frac 1\kappa}\int_0^t\frac{C}{(t-s)^{\frac34+\frac{1}{\kappa_1}-\frac 1\kappa}}\|\widetilde{u}(\ot)(s)
\cdot r^{-\frac{3}{10}}\ut(s)\|_{L^{\f{4\kappa_1}{4+\kappa_1}}(\Omega)}\,ds\\
&\leq t^{\frac {13}{20}-\frac 1\kappa}\int_0^t\frac{C}{(t-s)^{\frac34+\frac{1}{\kappa_1}-\frac 1\kappa}}
\|\widetilde{u}(\ot)(s)\|_{L^4(\Omega)}\|r^{-\frac{3}{10}}\ut(s)\|_{L^{\kappa_1}(\Omega)}\,ds\\
&\leq t^{\frac {13}{20}-\frac 1\kappa}\int_0^t\frac{C}{(t-s)^{\frac34+\frac{1}{\kappa_1}-\frac 1\kappa}}
\frac{\|\ot\|_{X_t}N_{\kappa_1}(T)}{s^{\frac{9}{10}-\frac{1}{\kappa_1}}}\,ds
\leq B^u_{\kappa,\kappa_1}\|\ot\|_{X_T}N_{\kappa_1}(T),
\end{split}
\end{equation}
where the exponents $\kappa\in]{20}/{13},\infty],~\kappa_1\in ]{20}/{13},\infty]$ satisfying
\begin{equation}\label{fix18}
\frac{4\kappa_1}{4+\kappa_1}\leq \kappa \andf \frac34+\frac{1}{\kappa_1}-\frac1\kappa<1.
\end{equation}
Meanwhile it follows from \eqref{inteeqt} and \eqref{fix17} that
\begin{equation}\label{fix19}
N_\kappa(T)\leq J_{p,q,\kappa}(T)+B^u_{\kappa,\kappa_1}\|\ot\|_{X_T}N_{\kappa_1}(T).
\end{equation}
Note that $J_{p,q,\kappa}(T),~\|\ot\|_{X_T},~N_2(T)\rightarrow0$ as $T\rightarrow0$,
by taking $\kappa_1=2$ in \eqref{fix19}, it follows from \eqref{fix18} that for any
$\kappa\in ]{20}/{13},4[$, we have
\beq \label{NT0}
N_\kappa(T)\quad\mbox{ are bounded} \andf N_\kappa(T)\rightarrow0\quad \mbox{as}\  T\rightarrow0.
\eeq
Next, taking $\kappa_1=3$ in \eqref{fix19}, we deduce from \eqref{fix18} that \eqref{NT0} holds  for any $\kappa\in [4,12[$.
Along the same line,  taking $\kappa_1=10$ in \eqref{fix19} ensures \eqref{NT0} for any $\kappa\in [12,\infty]$.
Hence we prove that \eqref{NT0} holds for any
$\kappa\in ]{20}/{13},\infty]$.

To handle $L_p(T)$,  we get, by a similar derivation of \eqref{fix6}, that
but here we need to split the integral area in two parts,
\begin{equation}\label{fix20}\begin{split}
&t^{1-\frac 1{p}}\|\mathcal{F}^\omega\bigl((\omega,\ut),(\omega,\ut)\bigr)(t)\|_{L^p(\Omega)}\\
&\leq  C\, t^{1-\frac 1{p}}\int_0^{\frac t2}\frac{\bigl(s^{\frac 14}\|\ot(s)\|_{L^{\frac 43}(\Omega)}\bigr)^2}{(t-s)^{\frac 32-\frac 1{p}}s^{\frac 12}}
+\frac{\bigl(s^{\frac {3}{20}}\|r^{-\frac{3}{10}}\ut(s)\|_{L^{2}(\Omega)}\bigr)^2}{(t-s)^{\frac 32-\frac 1{p}+\frac15}s^{\frac {3}{10}}}\,ds\\
&\qquad+C\,  t^{1-\frac 1{p}}\int_{\frac t2}^t\Bigl(\frac{\bigl(s^{1-\frac{1}{p_1}}\|\ot(s)\|_{L^{p_1}(\Omega)}\bigr)^2
}{(t-s)^{\frac{2}{p_1}-\frac 1{p}}s^{2-\frac{2}{p_1}}}
+\frac{\bigl(s^{\frac{13}{20}-\frac{1}{2p}}\|r^{-\frac{3}{10}}\ut(s)\|_{L^{2p}(\Omega)}\bigr)^2}
{(t-s)^{\frac{7}{10}}s^{\frac{13}{10}-\frac{1}{p}}}\Bigr)ds,\\
&\leq A^\omega_{p,p_1,p_2}\bigl(\|\ot\|^2_{X_T}+\|\ut\|^2_{X_T}+L_{p_1}^2(T)+N^2_{2p}(T)\bigr).
\end{split}
\end{equation}
where the exponents $p,~p_1\in]1,\infty]$ satisfying
\begin{equation}\label{fix21}
\frac 12\leq\frac{2}{p_1}-\frac 1p<1.
\end{equation}
Then we deduce from \eqref{inteeqt} and \eqref{fix20} that
\begin{equation}\label{fix22}
L_p(T)\leq J_{p,q,\kappa}(T)+A^\omega_{p,p_1,p_2}\bigl(\|\ot\|^2_{X_T}+\|\ut\|^2_{X_T}+L_{p_1}^2(T)+N^2_{2p}(T)\bigr).
\end{equation}
Note that $J_{p,q,\kappa}(T),~L_{\frac43}=\|\ot\|_{X_T},~\|\ut\|_{X_T},~N_{2p}(T)\rightarrow0$ as $T\rightarrow0$,
by taking $p_1=\frac 43$ in \eqref{fix22}, we deduce from \eqref{fix21}, that
\beq \label{LP0}
L_p(T) \quad \mbox{are bounded} \andf L_p(T)\rightarrow0 \quad\mbox{as} \ T\rightarrow0
\eeq for any
$p \in]1,2[$.
Next taking $p_1=\frac 53$ in  \eqref{fix22} ensures that \eqref{LP0} holds for any $p\in[2,5[$.
Similarly, taking $p_1=\frac 52$  in  \eqref{fix22} implies that \eqref{LP0} holds  for any $p\in[5,\infty]$.
Thus \eqref{LP0} holds for any
$p\in]1,\infty]$.
This completes the proof of Theorem \ref{thmlocal}.
\end{proof}

Let $L_p(T),~M_q(T),~N_\kappa(T)$  be given by \eqref{LMN1}, for any $s\in\N^+$, we shall denote
\begin{equation}\label{LMN2}\begin{split}
&L_{p_1,p_2}(T)\eqdefa\sup_{p_1\leq p\leq p_2}L_p(T),\quad
M_{q_1,q_2}(T)\eqdefa\sup_{q_1\leq q\leq q_2} M_q(T),\\
&N_{\kappa_1,\kappa_2}(T)\eqdefa\sup_{\kappa_1\leq \kappa\leq \kappa_2} N_\kappa(T) \andf L^s_{p_1,p_2}(T)\eqdefa L_{p_1,p_2}(T)+|L_{p_1,p_2}(T)|^s,\\
&M^s_{q_1,q_2}(T)\eqdefa M_{q_1,q_2}(T)+|M_{q_1,q_2}(T)|^s,\quad  N^s_{\kappa_1,\kappa_2}(T)\eqdefa N_{\kappa_1,\kappa_2}(T)+|N_{\kappa_1,\kappa_2}(T)|^s.
\end{split}\end{equation}

For later use, we state the following result.
\begin{cor}
{\sl For any $0\leq\gamma\leq\delta\leq 1,~1\leq p\leq\infty,~1\leq q_2\leq q_1\leq\infty$, under the assumptions of Theorem \ref{thmlocal},
if we   assume moreover that $r^{-\gamma}\ut_0\in L^{q_2}(\Omega)$, then there hold
\ben\label{cor3.1}
&&\|r^{-\delta}\ut(t)\|_{L^{q_1}(\Omega)}
\lesssim\frac{L_{\frac{20}{17},2}(T)N_{\frac{20}{13},20}(T)}{t^{\frac12+\frac{\delta}{2}-\frac1{q_1}}}
+\frac{\|r^{-\gamma}\ut_0\|_{L^{q_2}(\Omega)}}{t^{\frac{\delta-\gamma}{2}+\frac1{q_2}-\frac1{q_1}}},\quad
\forall\ t\in [0, T],\\
&&
\|r^{-\delta}\ot(t)\|_{L^p(\Omega)}
\lesssim\frac{L_{1,\frac{20}{3}}^5(T)+M_2^3(T)
+N^5_{\frac{20}{13},20}(T) }{t^{1+\frac{\delta}2-\frac1p}},\quad
\forall \ t\in [0, T]. \label{cor3.2}
\een
}
\end{cor}

\begin{proof}
 When $q_1\in[1,10]$,  in view of \eqref{inteeqt}, we get, by applying Proposition \ref{lemsemi} and then Lemma \ref{lem2.2.2}, that
\begin{align*}
\|r^{-\delta}&\ut(t)\|_{L^{q_1}(\Omega)}
\lesssim \frac{\|r^{-\gamma}\ut_0\|_{L^{q_2}(\Omega)}}{t^{\frac{\delta-\gamma}{2}+\frac1{q_2}-\frac1{q_1}}}
+\int_0^{\frac t2} \frac{\|\widetilde{u}(s)\|_{L^{\frac{20}{7}}(\Omega)}
\|r^{-\frac{3}{10}}\ut(s)\|_{L^{\frac{20}{13}}(\Omega)}}{(t-s)^{\frac{27}{20}+\frac{\delta}{2}-\frac{1}{q_1}}}\, ds\\
&\qquad\qquad\qquad +\int_{\frac t2}^t \frac{\|\widetilde{u}(s)\|_{L^{\frac{20}{7}q_1}(\Omega)}\|r^{-\frac{3}{10}}\ut(s)\|_{L^{{\frac{20}{13}q_1}}(\Omega)}}
{(t-s)^{\frac{7}{20}+\frac{\delta}{2}}}\,ds\\
&\lesssim \frac{\|r^{-\gamma}\ut_0\|_{L^{q_2}(\Omega)} }{t^{\frac{\delta-\gamma}{2}+\frac1{q_2}-\frac1{q_1}}}
+\int_0^{\frac t2} \frac{L_{\frac{20}{17}}(T)N_{\frac{20}{13}}(T)}{(t-s)^{\frac{27}{20}+\frac{\delta}{2}-\frac{1}{q_1}}s^{\frac{3}{20}}}\,ds
+\int_{\frac t2}^t \frac{L_{\frac{20q_1}{10q_1+7}}(T)N_{\frac{20}{13}q_1}(T)}{(t-s)^{\frac{7}{20}+\frac{\delta}{2}}s^{\frac{23}{20}-\frac1{q_1}}}\,ds\\
&\lesssim \frac{\|r^{-\gamma}\ut_0\|_{L^{q_2}(\Omega)} }{t^{\frac{\delta-\gamma}{2}+\frac1{q_2}-\frac1{q_1}}}
+\frac{L_{\frac{20}{17}}(T)N_{\frac{20}{13}}(T)+L_{\frac{20q_1}{10q_1+7}}(T)N_{\frac{20}{13}q_1}(T)}
{t^{\frac12+\frac{\delta}{2}-\frac1{q_1}}},
\end{align*}
which yields \eqref{cor3.1} for $q_1\in[1,10]$.

When $q_1\in]10,\infty]$, we get, by a similar derivation, that
\begin{align*}
\|r^{-\delta}&\ut(t)\|_{L^{q_1}(\Omega)}
\lesssim \frac{\|r^{-\gamma}\ut_0\|_{L^{q_2}(\Omega)} }{t^{\frac{\delta-\gamma}{2}+\frac1{q_2}-\frac1{q_1}}}
+\int_0^{\frac t2} \frac{L_{\frac{20}{17}}(T)N_{\frac{20}{13}}(T)}{(t-s)^{\frac{27}{20}+\frac{\delta}{2}-\frac{1}{q_1}}s^{\frac{3}{20}}}\,ds
+\int_{\frac t2}^t \frac{L_{2}(T)N_{20}(T)}{(t-s)^{\frac{2}{5}+\frac{\delta}{2}-\frac1{q_1}}s^{\frac{11}{10}}}\,ds\\
&\lesssim \frac{\|r^{-\gamma}\ut_0\|_{L^{q_2}(\Omega)} }{t^{\frac{\delta-\gamma}{2}+\frac1{q_2}-\frac1{q_1}}}
+\frac{L_{\frac{20}{17}}(T)N_{\frac{20}{13}}(T)+L_{2}(T)N_{20}(T)}
{t^{\frac12+\frac{\delta}{2}-\frac1{q_1}}}.
\end{align*}
This proves the estimate \eqref{cor3.1}.

To handle the estimate \eqref{cor3.2}, we first consider the case when $\delta\in[0,\frac12]$ and $p\in[1,5]$. In this case,
 applying Proposition \ref{lemsemi} to  \eqref{inteeqt} and then using the estimate \eqref{cor3.1} gives rise to
\begin{equation}\begin{split}\label{cor3.3}
\|r^{-\delta}\ot(t)\|_{L^p(\Omega)}
\lesssim& \frac{\|\ot_0\|_{L^1(\Omega)}}{t^{\frac{\delta}{2}+1-\frac1p}}
+\int_0^{\frac t2}\Bigl(\frac{\|\widetilde{u}(s)\|_{L^{4}(\Omega)}\|\ot(s)\|_{L^{\frac{4}{3}(\Omega)}}}{(t-s)^{\frac32+\frac{\delta}{2}-\frac1p}}
+\frac{\bigl\|{\ut}(s)\bigr\|_{L^{2}(\Omega)}^2}{(t-s)^{2+\frac{\delta}{2}-\frac1p}}\Bigr)\,ds\\
&+\int_{\frac t2}^t \Bigl(\frac{\|\widetilde{u}(s)\|_{L^{4p}(\Omega)}\|\ot(s)\|_{L^{\frac{4p}{3}(\Omega)}}}{(t-s)^{\frac12+\frac{\delta}{2}}}
+\frac{\bigl\|r^{-\frac{1+\delta}{2}}{\ut}(s)\bigr\|_{L^{2p}(\Omega)}^2}{(t-s)^{\frac{1}{2}}}\Bigr)\,ds\\
\lesssim& \frac{\|\ot_0\|_{L^1(\Omega)}}{t^{\frac{\delta}{2}+1-\frac1p}}
+\int_0^{\frac t2}\Bigl(\frac{L_{\frac43}^2(T)}{(t-s)^{\frac32+\frac{\delta}{2}-\frac1p}s^{\frac{1}2}}
+\frac{M_2^2(T)}{(t-s)^{2+\frac{\delta}{2}-\frac1p}}\Bigr)\,ds\\
&+\int_{\frac t2}^t \Bigl(\frac{L_{\frac{4p}{1+2p}}(T)L_{\frac{4p}{3}}(T)}{(t-s)^{\frac12+\frac{\delta}{2}}s^{\frac{3}2-\frac 1p}}
+\frac{L^2_{\frac{20}{17},\frac{20}{11}}(T)N^2_{\frac{20}{13},20}(T)+\|\ut_0\|_{L^2(\Omega)}^2}
{(t-s)^{\frac{1}{2}}s^{\frac{3}2+\frac{\delta}{2}-\frac 1p}}\Bigr)\,ds\\
\lesssim& \frac{L_{1,\frac{20}{3}}^4(T)+M_2^2(T)
+N^4_{\frac{20}{13},20}(T) }{t^{\frac{\delta}{2}+1-\frac1p}},
\end{split}\end{equation}
where in the last step we use the fact that $\|\ot_0\|_{L^1(\Omega)}\leq L_1,~\|\ut_0\|_{L^2(\Omega)}\leq M_2.$

 On the other side, when $p\in]5,\infty]$, we have
\begin{equation}\begin{split}\label{cor3.4}
\|r^{-\delta}\ot(t)\|_{L^p(\Omega)}
\lesssim &\frac{\|\ot_0\|_{L^1(\Omega)}}{t^{\frac{\delta}{2}+1-\frac1p}}
+\int_0^{\frac t2}\Bigl(\frac{\|\widetilde{u}(s)\|_{L^{4}(\Omega)}\|\ot(s)\|_{L^{\frac{4}{3}(\Omega)}}}{(t-s)^{\frac32+\frac{\delta}{2}-\frac1p}}
+\frac{\bigl\|{\ut}(s)\bigr\|_{L^{2}(\Omega)}^2}{(t-s)^{2+\frac{\delta}{2}-\frac1p}}\Bigr)\,ds\\
&+\int_{\frac t2}^t \Bigl(\frac{\|\widetilde{u}(s)\|_{L^{20}(\Omega)}\|\ot(s)\|_{L^{\frac{20}{3}(\Omega)}}}{(t-s)^{\frac{7}{10}+\frac{\delta}{2}-\frac1p}}
+\frac{\bigl\|r^{-\frac{1+\delta}{2}}{\ut}(s)\bigr\|_{L^{2p}(\Omega)}^2}{(t-s)^{\frac{1}{2}}}\Bigr)\,ds\\
&\lesssim \frac{L_{1,\frac{20}{3}}^4(T)+M_2^2(T)
+N^4_{\frac{20}{13},20}(T) }{t^{\frac{\delta}{2}+1-\frac1p}}.
\end{split}\end{equation}
Combining \eqref{cor3.3} with \eqref{cor3.4} leads to the  estimate \eqref{cor3.2} for $\delta\in[0,\frac12]$.

For the remaining case when $\delta\in]1/2,1]$, we first get, by a similar derivation of \eqref{cor3.3}  and then using \eqref{cor3.1} that
for $p\in[1,5]$
\begin{equation}\begin{split}\label{cor3.5}
\|r^{-\delta}\ot(t)\|_{L^p(\Omega)}
\lesssim& \frac{\|\ot_0\|_{L^1(\Omega)}}{t^{\frac{\delta}{2}+1-\frac1p}}
+\int_0^{\frac t2}\Bigl(\frac{\|\widetilde{u}(s)\|_{L^{4}(\Omega)}\|\ot(s)\|_{L^{\frac{4}{3}(\Omega)}}}{(t-s)^{\frac32+\frac{\delta}{2}-\frac1p}}
+\frac{\bigl\|{\ut}(s)\bigr\|_{L^{2}(\Omega)}^2}{(t-s)^{2+\frac{\delta}{2}-\frac1p}}\Bigr)\,ds\\
&+\int_{\frac t2}^t \Bigl(\frac{\|\widetilde{u}(s)\|_{L^{4p}(\Omega)}\|r^{-\frac12}\ot(s)\|_{L^{\frac{4p}{3}(\Omega)}}}{(t-s)^{\frac14+\frac{\delta}{2}}}
+\frac{\bigl\|r^{-\frac{1+\delta}{2}}{\ut}(s)\bigr\|_{L^{2p}(\Omega)}^2}{(t-s)^{\frac{1}{2}}}\Bigr)\,ds\\
&\lesssim \frac{L_{1,\frac{20}{3}}^5(T)+M_2^3(T)
+N^5_{\frac{20}{13},20}(T) }{t^{\frac{\delta}{2}+1-\frac1p}}.
\end{split}\end{equation}

Finally when  $p\in]5,\infty]$, we  deduce, by a similar derivation of \eqref{cor3.4}, that
\begin{equation}\begin{split}\label{cor3.6}
\|r^{-\delta}\ot(t)\|_{L^p(\Omega)}
\lesssim& \frac{\|\ot_0\|_{L^1(\Omega)}}{t^{\frac{\delta}{2}+1-\frac1p}}
+\int_0^{\frac t2}\Bigl(\frac{\|\widetilde{u}(s)\|_{L^{4}(\Omega)}\|\ot(s)\|_{L^{\frac{4}{3}(\Omega)}}}{(t-s)^{\frac32+\frac{\delta}{2}-\frac1p}}
+\frac{\bigl\|{\ut}(s)\bigr\|_{L^{2}(\Omega)}^2}{(t-s)^{2+\frac{\delta}{2}-\frac1p}}\Bigr)\,ds\\
&+\int_{\frac t2}^t \Bigl(\frac{\|\widetilde{u}(s)\|_{L^{20}(\Omega)}\|r^{-\frac12}\ot(s)\|_{L^{\frac{20}{3}(\Omega)}}}
{(t-s)^{\frac{9}{20}+\frac{\delta}{2}-\frac1p}}
+\frac{\bigl\|r^{-\frac{1+\delta}{2}}{\ut}(s)\bigr\|_{L^{2p}(\Omega)}^2}{(t-s)^{\frac{1}{2}}}\Bigr)\,ds\\
&\lesssim \frac{L_{1,\frac{20}{3}}^5(T)+M_2^3(T)
+N^5_{\frac{20}{13},20}(T) }{t^{\frac{\delta}{2}+1-\frac1p}}.
\end{split}\end{equation}
Combining \eqref{cor3.5} with \eqref{cor3.6} implies \eqref{cor3.2} for $\delta\in]\frac12,1]$. This completes the proof of the estimate
\eqref{cor3.2} and hence the corollary.
\end{proof}

\section{Global {\it a priori} estimates of \eqref{otut} with nearly critical initial data}\label{sec4}

The goal of this section is basically to prove that, as long as the initial data
belongs to the almost critical spaces, the system has a unique global solution.

Let us introduce another two variables which are of great importance in our work, namely
\begin{equation}\label{defetaV}
\eta\eqdefa\frac{\ot}{r},\quad V^\varepsilon\eqdefa \frac{\ut}{r^{1-\varepsilon}}\quad \mbox{for any}\ \varepsilon\in]0,1[.
\end{equation}
And it is not difficult to deduce the equations for $\eta$ and $V^\varepsilon$
from \eqref{otut} that
\begin{equation}\label{etaV}
\left\{
\begin{array}{l}
\displaystyle \pa_t \eta+(u^r\pa_r+u^z\pa_z) \eta-(\Delta+\frac 2r\pa_r)\eta
-\frac{2V \pa_z V}{r^{2\varepsilon}}=0,\\
\displaystyle \pa_t V+(u^r\pa_r+u^z\pa_z) V+(2-\varepsilon)\frac{u^r V}{r}
-(\Delta+ 2(1-\varepsilon)\frac {\pa_r}{r})V
+(2\varepsilon-\varepsilon^2)\frac{V}{r^2}=0,\\
\displaystyle \eta|_{t=0} =\eta_0=\frac{\ot_0}{r},\, V|_{t=0} =V_0=\frac{\ut_0}{r^{1-\varepsilon}},
\end{array}
\right.
\end{equation}
here and in all that follows, we always denote $V^\varepsilon$ as $V$,
if there is no ambiguity.

\begin{prop}\label{nearglobal}
{\sl Let $(u^r,u^\th,u^z)$ be a smooth enough solution of \eqref{1.2} on $[0,T].$
Let $p\in]1,\f {21}{20}]$,~$\e=\frac{-9p^2+21p-4}{24p-2}\in\bigl[\frac{3251}{9280},\frac{4}{11}\bigr[$, and $q$ be given by
\begin{equation}\label{nearglobal1}
(2-\e)q=3p.
\end{equation}
We assume that the initial data $\eta_0\in L^p$, $V_0^\e=\f{u^\th}{r^{1-\e}}\in L^q$,
$r\ut_0\in L^\infty\bigcap L^{\frac{1}{p-1}},$ which satisfy
\begin{equation}\label{nearglobal2}
(2M_0)^{\frac{3(p+2)}{p(3p+11)}}
\|r\ut_0\|_{L^{\frac{10(12p-1)}{9(p-1)(p+2)(p+3)}}}^{\f{10(12p-1)}{3p(p+3)(3p+11)}}
+(2M_0)^{\frac{p-1}{4p}}
\|r\ut_0\|_{L^{\frac{2}{3(p-1)^2}}}^{\f1{6p}}\leq c_0(p-1),
\end{equation}
for some sufficiently small  constant $c_0$ which does not depend on the choice of $p$,
and $M_0\eqdefa\|V_0\|_{L^q}^q+\|\eta_0\|_{L^p}^p$.
Then for any $t\in [0,T]$,  we have
\begin{equation}\begin{split}\label{nearglobal3}
\|V(t)\|_{L^q}^q&+\|\eta(t)\|_{L^p}^p
+\frac{p-1}{2}\|\nabla|\eta|^{\frac p2}\|_{L^2((0,t)\times\R^3)}^2
\\
&+\frac12\Bigl(\|\nabla|V|^{\frac q2}\|_{L^2((0,t)\times\R^3)}^2+\Bigl\|\frac{|V|^{\frac q2}}{r}\Bigr\|_{L^2((0,t)\times\R^3)}^2\Bigr)
\leq 2\bigl(\|V_0\|_{L^q}^q+\|\eta_0\|_{L^p}^p\bigr).
\end{split}\end{equation}
}
\end{prop}

Let us remark that both the index $\frac{10(12p-1)}{9(p-1)(p+2)(p+3)}$ and ${\frac{2}{3(p-1)^2}}$
are close enough to $\infty$ as long as $p$ approaches $1,$ which corresponds to the case with initial data in the critical spaces.

The proof of the above proposition relies on the following lemmas:

\begin{lem}\label{lem3.1}
{\sl Under the assumptions of Proposition \ref{nearglobal}, for any $t\in ]0,T],$ there holds
\begin{equation}\label{3.11}\begin{split}
&\frac1{q}\frac{d}{dt}\|V(t)\|_{L^{q}}^{q}+
\frac{4(q-1)}{q^2}\bigl\|\nabla|V|^{\frac q2}\bigr\|_{L^2}^2+(2\e-\e^2)\int_{\R^3}\frac{|V|^{q}}{r^2}\,dx\\
&\qquad\lesssim\|r\ut\|_{L^{\f{5}{2-\e}\cdot\frac{p}{(p-1)(p+2)}}}^{\f{15}{(2-\e)(3p+11)}}
\||V|^{\f q2}\|_{L^2}^{\f{23-17p-3p^2}{p(3p+11)}}
\|\eta\|_{L^p}^{\frac{3p^2+23p-11}{2(3p+11)}}\\
&\qquad\qquad\qquad\qquad\times
 \Bigl(\|\na|\eta|^{\f{p}2}\|_{L^{2}}^{2}+
\|\na|V|^{\f q2}\|_{L^2}^{2}+
\int_{\R^3}\f{|V|^{q}}{r^2}\,dx\Bigr).
\end{split}\end{equation}
}
\end{lem}

\begin{lem}\label{lem3.2}
{\sl Under the assumptions of Proposition \ref{nearglobal}, for any $t\in ]0,T],$ there holds
\begin{equation}\label{3.19}\begin{split}
&\frac1p\f{d}{dt}\|\eta(t)\|_{L^p}^p
+\frac{4(p-1)}{p^2}\|\nabla|\eta|^{\frac p2}\|_{L^2}^2\\
&\qquad\lesssim \|r\ut\|_{L^{\frac{2}{3(p-1)^2}}}^{\f1{6p}}
\|\eta\|^{\frac{p-1}4}_{L^p}\Bigl(\|\nabla|\eta|^{\frac p2}\|_{L^2}^2+\|\nabla|V|^{\frac q2}\|_{L^2}^2
+\Bigl\|\frac{|V|^q}{r^2}\Bigr\|_{L^1}\Bigr).
\end{split}\end{equation}}
\end{lem}

Let us admit the above lemmas and continue our proof of Proposition \ref{nearglobal}.

\begin{proof}[Proof of Proposition \ref{nearglobal}]
By virtue of Lemma \ref{lem2.1.2}, we get, by summarizing
\eqref{3.11} for $\e=\frac{-9p^2+21p-4}{24p-2}$ with \eqref{3.19},  that
\begin{equation}\begin{split}\label{3.20}
&\f{d}{dt}\bigl(\|V(t)\|_{L^q}^q+\|\eta(t)\|_{L^p}^p\bigr)
+\frac{4(p-1)}{p}\|\nabla|\eta|^{\frac p2}\|_{L^2}^2\\
&\qquad\qquad+\frac{4(q-1)}{q}\|\nabla|V|^{\frac q2}\|_{L^2}^2
+\frac{3p(-9p^2+21p-4)}{24p-2}\Bigl\|\frac{|V|^q}{r^2}\Bigr\|_{L^1}\\
&\leq C
\Bigl(\|r\ut_0\|_{L^{\frac{10(12p-1)}{9(p-1)(p+2)(p+3)}}}^{\f{10(12p-1)}{3p(p+3)(3p+11)}}
\||V|^{\f q2}\|_{L^2}^{\f{23-17p-3p^2}{p(3p+11)}}
\|\eta\|_{L^p}^{\frac{3p^2+23p-11}{2(3p+11)}}
+\|r\ut_0\|_{L^{\frac{2}{3(p-1)^2}}}^{\f1{6p}}
\|\eta\|^{\frac{p-1}4}_{L^p}\Bigr)\\
&\qquad\qquad\qquad\qquad\qquad\qquad\qquad\qquad\qquad\times \Bigl(\|\nabla|\eta|^{\frac p2}\|_{L^2}^2+\|\nabla|V|^{\frac q2}\|_{L^2}^2
+\Bigl\|\frac{|V|^q}{r^2}\Bigr\|_{L^1}\Bigr).
\end{split}\end{equation}
Let $M_0\eqdefa\|V_0\|_{L^q}^q+\|\eta_0\|_{L^p}^p$, and let $T'>0$ be determined by
\begin{equation}\label{3.30}
T'\eqdefa\sup\Bigl\{\ t\in [0,T]:\  \sup_{t'\in [0,t]}\bigl(\|V(t)\|_{L^q}^q+\|\eta(t)\|_{L^p}^p\bigr)
\leq 2M_0\ \Bigr\}.
\end{equation}
If $T'<T,$ then   for $t\leq T',$  we deduce  from the \eqref{3.20} that
\begin{equation*}\begin{split}
&\f{d}{dt}\bigl(\|V(t)\|_{L^q}^q+\|\eta(t)\|_{L^p}^p\bigr)
+(p-1)\|\nabla|\eta|^{\frac p2}\|_{L^2}^2+\|\nabla|V|^{\frac q2}\|_{L^2}^2
+\Bigl\|\frac{|V|^q}{r^2}\Bigr\|_{L^1}\\
&\quad \leq C
\Bigl((2M_0)^{\frac{3(p+2)}{p(3p+11)}}
\|r\ut_0\|_{L^{\frac{10(12p-1)}{9(p-1)(p+2)(p+3)}}}^{\f{10(12p-1)}{3p(p+3)(3p+11)}}
+(2M_0)^{\frac{p-1}{4p}}
\|r\ut_0\|_{L^{\frac{2}{3(p-1)^2}}}^{\f1{6p}}\Bigr)\\
&\qquad\qquad\qquad\qquad\qquad\qquad\quad\times\Bigl(\|\nabla|\eta|^{\frac p2}\|_{L^2}^2+\|\nabla|V|^{\frac q2}\|_{L^2}^2
+\Bigl\|\frac{|V|^q}{r^2}\Bigr\|_{L^1}\Bigr),
\end{split}\end{equation*}
which together with  \eqref{nearglobal2} ensures that
 for any $t$ in $[0,T']$,
\begin{equation}
\f{d}{dt}\bigl(\|V(t)\|_{L^q}^q+\|\eta(t)\|_{L^p}^p\bigr)
+\frac{p-1}2\|\nabla|\eta|^{\frac p2}\|_{L^2}^2+\frac12\|\nabla|V|^{\frac q2}\|_{L^2}^2
+\frac12\Bigl\|\frac{|V|^q}{r^2}\Bigr\|_{L^1}\leq0.
\end{equation}
This  in particular gives rise to
\begin{equation}\label{3.32}
\|V(t)\|_{L^q}^q+\|\eta(t)\|_{L^p}^p
\leq M_0\quad \mbox{for any}\quad t\leq T'.
\end{equation}
This contradicts with the definition of $T'$ given by \eqref{3.30}. As a result, it comes out $T'=T$, and there holds
\eqref{nearglobal3} for any $t\in [0,T].$ this completes the proof of the proposition.
\end{proof}

Let us now turn to the proof of Lemmas \ref{lem3.1} and \ref{lem3.2}.

\begin{proof}[Proof of Lemma \ref{lem3.1}]
For any
$p\in]1,\frac{21}{20}]$ and $q$ given by \eqref{nearglobal1}, we get, by multiplying
the second equation of \eqref{etaV} by $|V|^{q-2}V$ and then integrating the
resulting equality over $\R^3,$ that
$$\longformule{
\frac1{q}\frac{d}{dt}\|V(t)\|_{L^{q}}^{q}+\frac1{q}\int_{\R^3}(u^r\pa_r+u^z\pa_z)|V|^{q}\,dx+(2-\e)\int_{\R^3}\f{u^r}r|V|^{q}\,dx}{{}
-\int_{\R^3}\D V | |V|^{q-2}V\,dx-\frac{2(1-\e)}{q}\int_{\R^3}\pa_r|V|^{q}\frac1r\,dx+(2\e-\e^2)\int_{\R^3}\frac{|V|^{q}}{r^2}\,dx=0.
}
$$
Using the fact that $\pa_r(ru^r)+\pa_z(ru^z)=0$,
which implies $\int_{\R^3}(u^r\pa_r+u^z\pa_z)|V|^{q}\,dx=0$, and the homogeneous Dirichlet boundary condition for $u^r$ on $r=0,$ we deduce
\begin{equation*}
\begin{split}
\frac1{q}\frac{d}{dt}\|V(t)\|_{L^{q}}^{q}+\bigl(2-\e\bigr)&\int_{\R^3}\f{u^r}r|V|^{q}\,dx
+\frac{2(1-\e)}{q}\int_{-\infty}^{+\infty} |V|^{q}\big|_{r=0}\,dz\\
&+\frac{4(q-1)}{q^2}\bigl\|\nabla|V|^{\frac q2}\bigr\|_{L^2}^2+(2\e-\e^2)\int_{\R^3}\frac{|V|^{q}}{r^2}\,dx=0,
\end{split}
\end{equation*}
which implies
\begin{equation}\label{3.9}
\frac1{q}\frac{d}{dt}\|V(t)\|_{L^{q}}^{q}+
\frac{4(q-1)}{q^2}\bigl\|\nabla|V|^{\frac q2}\bigr\|_{L^2}^2+(2\e-\e^2)\int_{\R^3}\frac{|V|^{q}}{r^2}\,dx
\lesssim\int_{\R^3}\bigl|\f{u^r}r\bigr|\cdot|V|^{q}\,dx
\end{equation}

Let us take $q_1,q_2$ satisfying
$
 \frac{1}{q_1}+\frac{1}{q_2}=1,$
and temporarily assume that $q_1\in \bigl]\f{3p}{3-p},\infty\bigr[,$
so that it follows from Lemma \ref{lem2.6.1} and Sobolev embedding Theorem  that
\beno
\begin{split}
\int_{\R^3}\bigl|\frac{u^r}{r}\bigr|\cdot|V|^{q} dx\leq &\bigl\|\frac{u^r}{r}\bigr\|_{L^{q_1}}\||V|^{q}\|_{L^{q_2}}\\
\lesssim &\|\eta\|_{L^p}^{\frac{p-1}{2}+\frac{3p}{2q_1}}\|\eta\|_{L^{3p}}^{\frac{3-p}{2}-\frac{3p}{2q_1}}\||V|^{q}\|_{L^{q_2}}\\
\lesssim &\|\eta\|_{L^p}^{\frac{p-1}{2}+\frac{3p}{2q_1}}\|\na|\eta|^{\f{p}2}\|_{L^{2}}^{2\left(\frac{3-p}{2p}-\frac{3}{2q_1}\right)}\||V|^{q}\|_{L^{q_2}} .
\end{split}
\eeno
Take $q_2=\vartheta+\s+\f{2\s}{3p},$ with $\vartheta>0$ and $0<\s<1$ to be determined later, then we have
\begin{equation}
\begin{split}
\||V|^{q}\|_{L^{q_2}}=&\Bigl(\int_{\R^3}|V|^{q\vartheta}\Bigl(\f{|V|^{q}}{r^2}\Bigr)^{\s}|r^{2-\e}V|^{\f{2\s}{2-\e}}\,dx\Bigr)^{\f1{q_2}}\\
\leq &\|r\ut\|_{L^{\f{2\s}{2-\e}\cdot\frac{p}{p-1}}}^{\f{2\s}{q_2(2-\e)}}
\||V|^{\f q2}\|_{L^{\f{2p\vartheta}{1-p\sigma}}}^{\f{2\vartheta}{q_2}}\Bigl(\int_{\R^3}\f{|V|^{q}}{r^2}\,dx\Bigr)^{\f{\s}{q_2}}\\
\leq &\|r\ut\|_{L^{\f{2\s}{2-\e}\cdot\frac{p}{p-1}}}^{\f{2\s}{q_2(2-\e)}}
\||V|^{\f q2}\|_{L^2}^{\f{-\vartheta-3\s}{q_2}+\f{3}{pq_2}}\|\na|V|^{\f q2}\|_{L^2}^{\f{3(\vartheta+\s)}{q_2}-\f{3}{pq_2}}
\Bigl(\int_{\R^3}\f{|V|^{q}}{r^2}\,dx\Bigr)^{\f{\s}{q_2}},
\end{split}
\end{equation}
where in the last step, we use the interpolation inequality provided that
\begin{equation}
\f{2p\vartheta}{1-p\sigma}\in[2,6],
\end{equation}
which will be verified later. Then we get, by applying Young's inequality, that
\begin{equation}\label{3.10}
\begin{split}
\int_{\R^3}\bigl|\frac{u^r}{r}\bigr||V|^q dx \lesssim &\|r\ut\|_{L^{\f{2\s}{2-\e}\cdot\frac{p}{p-1}}}^{\f{2\s}{q_2(2-\e)}}
\||V|^{\f q2}\|_{L^2}^{\f{-\vartheta-3\s}{q_2}+\f{3}{pq_2}}
\|\eta\|_{L^p}^{\frac{p-1}{2}+\frac{3p}{2q_1}}\\
&\times\|\na|\eta|^{\f{p}2}\|_{L^{2}}^{2\left(\frac{3-p}{2p}-\frac{3}{2q_1}\right)}
\|\na|V|^{\f q2}\|_{L^2}^{\f{3(\vartheta+\s)}{q_2}-\f{3}{pq_2}}
\Bigl(\int_{\R^3}\f{|V|^{q}}{r^2}\,dx\Bigr)^{\f{\s}{q_2}}\\
\lesssim&\|r\ut\|_{L^{\f{2\s}{2-\e}\cdot\frac{p}{p-1}}}^{\f{2\s}{q_2(2-\e)}}
\||V|^{\f q2}\|_{L^2}^{\f{-\vartheta-3\s}{q_2}+\f{3}{pq_2}}
\|\eta\|_{L^p}^{\frac{p-1}{2}+\frac{3p}{2q_1}}\\
&\times
 \Bigl(\|\na|\eta|^{\f{p}2}\|_{L^{2}}^{2}+
\|\na|V|^{\f q2}\|_{L^2}^{2}+
\int_{\R^3}\f{|V|^{q}}{r^2}\,dx\Bigr),
\end{split}
\end{equation}
provided that
\beno
\f{3-p}{2p}-\f3{2q_1}+\f{3(\vartheta+\s)}{2q_2}-\f{3}{2pq_2}+\f\s{q_2}=1,
\eeno
and this can be satisfied by choosing
\begin{equation}\label{index}\begin{split}
&\sigma=\frac{5}{2(p+2)},\ \vartheta=1-\frac5{6p},\ \mbox{and}\ \  q_2=\vartheta+\s+\f{2\s}{3p}=\f{3p+11}{3(p+2)}.
\end{split}\end{equation}
It is easy to verify that for any $p\in\bigl]1,\frac{21}{20}\bigr]$,
the corresponding
$\f{2p\vartheta}{1-p\sigma}=\frac{2(p+2)(6p-5)}{3(4-3p)}$ is exactly in $[2,6]$,
and $q_2$ is exactly in $\bigl]1,\frac{3p}{4p-3}\bigr[$,
thus the conjugate number $q_1=\bigl(1-\frac{1}{q_2}\bigr)^{-1}>\f{3p}{3-p}$.
Thus all the above calculations go through.

With the indexes given by \eqref{index}, by inserting the Estimate \eqref{3.10} into \eqref{3.9},
we achieve \eqref{3.11}. This completes the proof of Lemma \ref{lem3.1}. \end{proof}

\begin{proof}[Proof of Lemma \ref{lem3.2}]
Analogue to the proof of Lemma \ref{lem3.1},
we get, by first multiplying the $\eta$ equation of \eqref{etaV} by $|\eta|^{p-2}\eta$ and then integrating the resulting
equality over $\R^3,$  that
$$\longformule{
\frac1p\f{d}{dt}\|\eta(t)\|_{L^p}^p+\f1p\int_{\R^3}(u^r\pa_r+u^z\pa_z)|\eta|^p\,dx}{{}-\int_{\R^3}\D \eta  | |\eta|^{p-2}\eta\,dx
-\f2p\int_{\R^3}\f{\pa_r|\eta|^p}{r}\,dx
=\int_{\R^3}\frac{2V \pa_z V}{r^{2\varepsilon}}|\eta|^{p-2}\eta dx.}
$$
Once again due to $\pa_r(ru^r)+\pa_z(ru^z)=0$ and $u^r|_{r=0}=0$, we get
\begin{equation}\label{3.12}
\frac1p\frac{d}{dt}\|\eta(t)\|_{L^p}^p
+\frac{4(p-1)}{p^2}\|\nabla|\eta|^{\frac p2}\|_{L^2_x}^2
+\frac 2p\int_{-\infty}^{+\infty}|\eta|^p|_{r=0} dz
=\int_{\R^3}\frac{2V \pa_z V}{r^{2\varepsilon}}|\eta|^{p-2}\eta dx.
\end{equation}
It is easy to observe that
\beno
\begin{split}
\bigl|\int_{\R^3}\frac{2V \pa_z V}{r^{2\varepsilon}}|\eta|^{p-2}\eta dx\bigr|\leq &\f2q\int_{\R^3}|\eta|^{p-1}\bigl|\pa_z|V|^{\f{q}2}\bigr|
\f{|V|^{2-\f{q}2}}{r^{2\e}}\,dx\\
\lesssim&\bigl\||\eta|^{p-1}\bigr\|_{L^\frac{2p}{p-1}}\bigl\|\pa_z|V|^{\frac q2}\bigr\|_{L^2}
\bigl\|\frac{|V|^{2-\frac q2}}{r^{2\varepsilon}}\bigr\|_{L^{2p}}.
\end{split}
\eeno
It follows from Sobolev embedding Theorem that
\beno
\begin{split}
\bigl\||\eta|^{p-1}\bigr\|_{L^\frac{2p}{p-1}}=\||\eta|^{\f{p}2}\|_{L^4}^{\f{2(p-1)}p}\lesssim &\||\eta|^{\f{p}2}\|_{L^2}^{\f{p-1}{2p}}
\|\na|\eta|^{\f{p}2}\|_{L^2}^{\f{3(p-1)}{2p}}\\
=&\|\eta\|_{L^p}^{\f{p-1}4}\|\na|\eta|^{\f{p}2}\|_{L^2}^{\f{3(p-1)}{2p}}.
\end{split}
\eeno
As a result, we obtain
\begin{equation}\label{3.13}\begin{split}
&\bigl|\int_{\R^3}\frac{2V \pa_z V}{r^{2\varepsilon}}|\eta|^{p-2}\eta dx\bigr|\lesssim
\|\eta\|^{\frac{p-1}4}_{L^p}\|\nabla|\eta|^{\frac p2}\|_{L^2}^\frac{3(p-1)}{2p}
\|\pa_z|V|^{\frac q2}\|_{L^2}
\Bigl\|\frac{|V|^{2-\frac q2}}{r^{2\varepsilon}}\Bigr\|_{L^{2p}}.
\end{split}\end{equation}
To handle the term $\Bigl\|\frac{|V|^{2-\frac q2}}{r^{2\varepsilon}}\Bigr\|_{L^{2p}},$  we split $\frac{|V|^{2-\frac q2}}{r^{2\varepsilon}}$ as
\begin{equation*}\label{3.22}
\frac{|V|^{2-\frac q2}}{r^{2\varepsilon}}=\left|\frac{|V|^q}{r^2}\right|^\alpha
\bigl|r^{2-\varepsilon}V\bigr|^{\f1{6p}} |V|^{\frac q2 \cdot\beta},
\end{equation*}
with $\alpha, \beta$ being determined by
\begin{equation*}\label{3.23}
\alpha=\f1{6p}\Bigl(1-\f\e2\Bigr)+\varepsilon \andf
\beta=\frac 2q\Bigl(2-\frac q2-q\alpha-\frac{1}{6p}\Bigr)=\f1{3p}\Bigl(7-\f2{3p}\Bigr)\Bigl(1-\f\e2\Bigr)-1-2\e.
\end{equation*}
It follows from Sobolev embedding Theorem that
\begin{equation}\label{3.14}
\begin{split}
\bigl\||V|^{\frac q2 \cdot\beta}\bigr\|_{L^{r}}\lesssim &
\bigl\|\nabla|V|^{\frac q2}\bigr\|_{L^2}^{\frac{3-p}{2p}-2\alpha}
\bigl\||V|^{\frac q2}\bigr\|_{L^2}^{2\alpha+\beta-\frac{3-p}{2p}}\with\\
\f1r=&\f{\beta}2+\f{2\alpha}3-\f{3-p}{6p}=\f1{18p}\Bigl(14-\f{23}2\e-\f{2-\e}p\Bigr)-\f13(1+\e).
\end{split}
\end{equation}
It is easy to verify that
$\alpha+\frac {1}{r}=\frac{1}{2p}-\frac{(p-1)^2}{4p}$ provided selecting
\begin{equation}\label{3.15}
\e=\frac{-9p^2+21p-4}{24p-2}\, ,
\end{equation}
which belongs to $\bigl[\frac{3251}{9280},\frac{4}{11}\bigr[$ whenever $p\in\bigl]1,\frac{21}{20}\bigr]$.

Moreover, under such choice of indexes, the term $\bigl\||V|^{\frac q2}\bigr\|_{L^2}^{2\alpha+\beta-\frac{3-p}{2p}}$
in \eqref{3.14} disappears. Then we get, by applying H\"{o}lder's inequality, that
\begin{equation*}\label{3.27}\begin{split}
\Bigl\|\frac{|V|^{2-\frac q2}}{r^{2\varepsilon}}\Bigr\|_{L^{2p}}
&\leq\Bigl\|\Bigl(\frac{|V|^q}{r^2}\Bigr)^{\alpha}\Bigr\|_{L^{\frac{1}{\alpha}}}
\bigl\||V|^{\frac q2 \cdot\beta}\bigr\|_{L^{r}}
\bigl\||r^{2-\varepsilon}V|^{\f1{6p}}\bigr\|_{L^{\frac{4p}{(p-1)^2}}}\\
&\leq\Bigl\|\frac{|V|^q}{r^2}\Bigr\|^{\alpha}_{L^1}
\bigl\|\nabla|V|^{\frac q2}\bigr\|_{L^2}^{\frac{3-p}{2p}-2\alpha}
\|r\ut\|_{L^{\frac{2}{3(p-1)^2}}}^{\f1{6p}}.\\
\end{split}\end{equation*}
Inserting the above inequality into \eqref{3.13} gives rise to
\beno
\begin{split}
\Bigl|\int_{\R^3}\frac{2V \pa_z V}{r^{2\varepsilon}}|\eta|^{p-2}\eta dx\Bigr|\lesssim
\|\eta\|^{\frac{p-1}4}_{L^p}\|\nabla|\eta|^{\frac p2}\|_{L^2}^\frac{3(p-1)}{2p}
\|\na|V|^{\frac q2}\|_{L^2}^{\frac{3-p}{2p}+1-2\alpha}
 \Bigl\|\frac{|V|^q}{r^2}\Bigr\|^{\alpha}_{L^1}
\|r\ut\|_{L^{\frac{2}{3(p-1)^2}}}^{\f1{6p}}.
\end{split}
\eeno
Note  that $\f{3(p-1)}{2p}+\bigl(\f{3-p}{2p}+1-2\alpha\bigr)+2\alpha=2,$ by substituting the above inequality into \eqref{3.12} and using Young's inequality,
 we achieve \eqref{3.19}. This completes the proof of Lemma \ref{lem3.2}. \end{proof}

\section{Global well-posedness with critical initial data}\label{sec5}
In what follows,
 we shall always denote $U\eqdefa\frac{\ut}{r}$ and $W\eqdefa r^{-\frac{7}{11}}\ut$.
\begin{lem}\label{lem5.1}
{\sl Under the assumption of Theorem \ref{thmmain}, for any $p\in[1,\f {21}{20}]$,
$\e=\frac{-9p^2+21p-4}{24p-2}$, and $q=\frac{3p}{2-\varepsilon}$, the local solutions constructed in Theorem \ref{thmlocal} satisfy
\ben
\|\eta(t)\|_{L^p}&\leq&\frac{C}{t^{\frac{3(p-1)}{2p}}}\cdot
\bigl(L_{1,\frac{20}{3}}^5(T)+M_2^3(T)
+N^5_{\frac{20}{13},20}(T)\bigr)\quad \forall\ t\in ]0, T],\label{5.1}\\
\label{5.2}
\|U(t)\|_{L^{\frac32}}&\leq& C\bigl( \|U_0\|_{L^{\frac32}}
+L_{\frac{20}{17},2}(T)N_{\frac{20}{13},20}(T)\bigr)\quad \forall\ t\in ]0, T],\\
\label{5.3}
\|V^\e(t)\|_{L^q}&\leq&\frac{C}{t^{\frac{3(p-1)}{2q}}}
\Bigl(\|U_0\|_{L^{\f32}}^{\f9{11}}\|r\ut_0\|_{L^\infty}^{\f2{11}}+
L_{\frac{20}{17},2}^2(T)+N^2_{\frac{20}{13},20}(T)\Bigr)\quad \forall\ t\in ]0, T].
\een
}
\end{lem}
\begin{proof}  Taking $\delta=1-\frac1p$ in \eqref{cor3.2} yields \eqref{5.1}. Likewise, taking $\delta=\gamma=\frac13$ and $~q_1=q_2=\frac32$  in \eqref{cor3.1} leads to \eqref{5.2}.

On the other hand, for any $p\in[1, {21}/{20}]$, due to the choice of $\e$ and $q$, we have
$$1-\e-\frac1q=\frac{9p^2-7}{24p-2}\geq\frac{1}{11},\quad q=\frac{24p-2}{3p+9}\geq\frac{11}{6}.$$
Noting that $\|V^\e(t)\|_{L^q}=\Bigl\|\frac{\ut}{r^{1-\e-\frac1q}}\Bigr\|_{L^q(\Omega)}$
and $\|r^{-\frac{7}{11}}\ut_0\|_{L^{\frac{11}{6}}}=\bigl\|r^{-\frac{1}{11}}\ut_0\bigr\|_{L^{\frac{11}{6}}(\Omega)}$,
then applying  \eqref{cor3.1} with $q_1=q,~q_2=\frac{11}{6},$  $\delta=1-\e-\frac1q$ and $\gamma=\frac{1}{11}$
gives
\begin{align*}
\|V^\e(t)\|_{L^q}&\leq\frac{C}{t^{\frac{3(p-1)}{2q}}}\cdot
\bigl(\|r^{-\frac{7}{11}}\ut_0\|_{L^{\frac{11}{6}}}+
L_{\frac{20}{17},2}(T)N_{\frac{20}{13},20}(T)\bigr)\\
&\leq\frac{C}{t^{\frac{3(p-1)}{2q}}}\cdot
\Bigl(\|U_0\|_{L^{\f32}}^{\f9{11}}\|r\ut_0\|_{L^\infty}^{\f2{11}}+
L_{\frac{20}{17},2}^2(T)+N^2_{\frac{20}{13},20}(T)\Bigr),\quad \forall t\in ]0, T],
\end{align*}
where we have used H\"{o}lder's inequality in the last step.
\end{proof}

\begin{proof}[Proof of Theorem \ref{thmmain}]
Due to $U_0=\frac{\ut_0}{r}\in L^{\frac32},~r\ut_0\in  L^\infty$, and
$$\|r^\kappa\ut_0\|_{L^{\frac{3}{1-\kappa}}}\leq\|U_0\|_{L^{\frac32}}^{\frac{1-\kappa}2}
\|r\ut_0\|_{L^\infty}^{\frac{1+\kappa}2},\ \forall \kappa\in]-1,1[,$$
we deduce  that both $\|\ut_0\|_{L^2(\Omega)}=\|r^{-\frac12}\ut_0\|_{L^{2}}$
and $\|r^{-\frac{3}{10}}\ut_0\|_{L^{\frac{20}{13}}(\Omega)}=\|r^{-\frac{19}{20}}\ut_0\|_{L^{\frac{20}{13}}}$
are sufficiently small as long as  $\|r\ut_0\|_{L^\infty}$ is small enough.
Then by Theorem \ref{thmlocal}, the equation \eqref{otut} has a unique mild solution
\begin{align*}
\ot\in C\bigl([0,T];&L^1(\Omega)\bigr)\bigcap C\bigl(]0,T];L^\infty(\Omega)\bigr),\
\ut\in C\bigl([0,T];L^2(\Omega)\bigr)\bigcap C\bigl(]0,T];L^\infty(\Omega)\bigr)\\
&\with  r^{-\frac{3}{10}}{\ut}\in C\bigl([0,T];L^{\frac{20}{13}}(\Omega)\bigr)\bigcap C\bigl(]0,T];L^\infty(\Omega)\bigr),
\end{align*}
and the lifespan $T>0$ depends only on $\omega_0^\theta$. We denote $t_0\eqdefa \frac T2$.
In the following, we will always abbreviate $L_p(T)$ as $L_p$, similar abbreviations for the remaining ones in
 \eqref{LMN1},~\eqref{LMN2}.

If  $r\ut(t_0)\in L^{A}\bigcap L^\infty$ and if  there exists some $p_0\in \left]1,\min\bigl(1+\frac{1}{10A},\frac{21}{20}\bigr)\right[$,
it follows from  Proposition \ref{nearglobal} that
for
$\e_0\eqdefa \frac{-9p_0^2+21p_0-4}{24p_0-2}$ and $q_0\eqdefa \frac{2(12p_0-1)}{3(p+3)},$  if there  holds
\begin{equation}\label{5.4}
(2M_1)^{\frac{3(p_0+2)}{p_0(3p_0+11)}}
\|r\ut(t_0)\|_{L^{\alpha(p_0)}}^{\f{10(12p_0-1)}{3p_0(p_0+3)(3p_0+11)}}
+(2M_1)^{\frac{p_0-1}{4p_0}}
\|r\ut(t_0)\|_{L^{\frac{2}{3(p_0-1)^2}}}^{\f1{6p_0}}\leq c_0\cdot(p_0-1),
\end{equation}
where $$M_1\eqdefa \|V^{\e_0}(t_0)\|_{L^{q_0}}^{q_0}+\|\eta(t_0)\|_{L^{p_0}}^{p_0} \andf
\alpha(p_0)\eqdefa \frac{10(12p_0-1)}{9(p_0-1)(p_0+2)(p_0+3)}.
$$
Then the system \eqref{otut} has a global solution. Moreover, in view of \eqref{nearglobal3} and Lemma \ref{lem5.1},
for all $t\in [t_0,\infty)$, there holds
\begin{equation}\label{5.5}
\|\eta(t)\|_{L^{p_0}}^{p_0}
+\frac{p_0-1}{2}\bigl\|\nabla|\eta|^{\frac {p_0}2}\bigr\|_{L^2([t_0,t)\times\R^3)}^2
\leq\frac{C}{t_0^{\frac{3(p_0-1)}{2}}}
\bigl(\|r\ut_0\|_{L^\infty}^{q_0}+L_{1,\frac{20}{3}}^{5p_0}+M_2^{3p_0}
+N^{5p_0}_{\frac{20}{13},20}\bigr).
\end{equation}
By the choice of $p_0$, $\alpha(p_0)>10A$. Then we deduce from Lemma \ref{lem2.1.2}
and Lemma \ref{lem5.1}  that
the smallness condition \eqref{5.4} holds provided that
\begin{equation}\begin{split}\label{5.6}
\|&r\ut_0\|_{L^\infty}\leq \frac{c_0}{L_{1,\frac{20}{3}}^6+M_2^4
+N^6_{\frac{20}{13},20}}
\min\Bigl\{\Bigl((p_0-1)\bigl(t_0^{\frac32}\|r\ut_0\|_{L^A}^{-A}\bigr)^{\frac{(p_0-1)^2}{4p_0}}\Bigr)
^{\frac{12p_0}{2-3A(p_0-1)^2}},\\
&\qquad\qquad\qquad\qquad\Bigl((p_0-1)
\bigl(t_0^{\frac32}\|r\ut_0\|_{L^A}^{-A}\bigr)^{\frac{3(p_0-1)(p_0+2)}{p_0(3p_0+11)}}\Bigr)
^{\f{3p_0(p_0+3)(3p_0+11)}{10(12p_0-1)-9A(p_0-1)(p_0+2)(p_0+3)}} \Bigr\}.
\end{split}\end{equation}

In the following, we consider estimates in critical spaces.
By a similar derivation as Lemma \ref{lem3.1}, that for $q_{1,1}, q_{2,1}$ satisfying
 $\frac1{q_{1,1}}+\frac1{q_{2,1}}=1$ and
$q_{1,1}>\frac{3p_0}{3-p_0}$, there holds
\begin{equation}\label{5.7}\begin{split}
\frac{6}{11}\frac{d}{dt}\|W(t)\|_{L^{\frac{11}{6}}}^{\frac{11}{6}}+&
\frac{120}{121}\bigl\|\nabla|W|^{\frac{11}{12}}\bigr\|_{L^2}^2+\frac{72}{121}\int_{\R^3}\frac{|W|^{\frac{11}{6}}}{r^2}\,dx
\lesssim\int_{\R^3}\bigl|\f{u^r}r\bigr|\cdot|W|^{\frac{11}{6}}\,dx\\
&\qquad\qquad\lesssim \|\eta\|_{L^{p_0}}^{\frac{p_0-1}{2}+\frac{3p_0}{2q_{1,1}}}
\|\na|\eta|^{\f{p_0}2}\|_{L^{2}}^{2\left(\frac{3-p_0}{2p_0}-\frac{3}{2q_{1,1}}\right)}\||W|^{\frac{11}{6}}\|_{L^{q_{2,1}}}\\
\end{split}\end{equation}
Take $q_{2,1}=\vartheta_1+\s_1+\f{2\s_1}{3},$ with $\vartheta_1>0,~0<\s_1<1$
to be determined later, then we have
\begin{equation*}
\begin{split}
\||W|^{\frac{11}{6}}&\|_{L^{q_{2,1}}}=
\Bigl(\int_{\R^3}|W|^{\frac{11}{6}\vartheta_1}\Bigl(\f{|W|^{\frac{11}{6}}}{r^2}\Bigr)^{\s_1}|r^{\frac{18}{11}}W|^{\f{11\s_1}{9}}\,dx\Bigr)^{\f1{q_{2,1}}}\\
\lesssim& \|r\ut\|_{L^{\f{11\s_1 p_0}{12(p_0-1)}}}^{\f{11\s_1}{9q_{2,1}}}
\||W|^{\f {11}{12}}\|_{L^{\f{6p_0\vartheta_1}{4-p_0-3p_0\sigma_1}}}^{\f{2\vartheta_1}{q_{2,1}}}
\Bigl(\int_{\R^3}\f{|W|^{\frac{11}{6}}}{r^2}\,dx\Bigr)^{\f{\s_1}{q_{2,1}}},\\
\lesssim& \|r\ut\|_{L^{\f{11\s_1 p_0}{12(p_0-1)}}}^{\f{11\s_1}{9q_{2,1}}}
\||W|^{\f{11}{12}}\|_{L^2}^{\f{-\vartheta_1-3\s_1}{q_{2,1}}+\f{4-p_0}{p_0q_{2,1}}}
\|\na|W|^{\f{11}{12}}\|_{L^2}^{\f{3(\vartheta_1+\s_1)}{q_{2,1}}-\f{4-p_0}{p_0q_{2,1}}}
\Bigl(\int_{\R^3}\f{|W|^{\frac{11}{6}}}{r^2}\,dx\Bigr)^{\f{\s_1}{q_{2,1}}},
\end{split}
\end{equation*}
where in the last step, we used Galiardo-Nirenberg  inequality provided that
\begin{equation}\label{5.8}
\f{6p_0\vartheta_1}{4-p_0-3p_0\sigma_1}\in[2,6],
\end{equation}
which will be verified later. Then we get, by applying Young's inequality, that
\begin{equation}\label{5.10}
\begin{split}
\int_{\R^3}\bigl|\frac{u^r}{r}\bigr||W|^{\frac{11}{6}} dx \lesssim &\|r\ut\|_{L^{\f{11\s_1 p_0}{12(p_0-1)}}}^{\f{11\s_1}{9q_{2,1}}}
\||W|^{\f{11}{12}}\|_{L^2}^{\f{-\vartheta_1-3\s_1}{q_{2,1}}+\f{4-p_0}{p_0q_{2,1}}}
\|\eta\|_{L^{p_0}}^{\frac{p_0-1}{2}+\frac{3p_0}{2q_{1,1}}}\\
&\times
 \Bigl(\|\na|\eta|^{\f{p_0}2}\|_{L^{2}}^{2}+
\|\na|W|^{\f{11}{12}}\|_{L^2}^{2}+
\int_{\R^3}\f{|W|^{\frac{11}{6}}}{r^2}\,dx\Bigr),
\end{split}
\end{equation}
provided that
\beno
\f{3-p_0}{2p_0}-\f3{2q_{1,1}}+\f{3(\vartheta_1+\s_1)+1}{2q_{2,1}}-\f{2}{p_0q_{2,1}}+\f{\s_1}{q_{2,1}}=1,
\eeno
which in particular holds by taking
\begin{equation}\label{index0}
\sigma_1=\frac{2}{5},\ \vartheta_1=\frac23,\ \mbox{and}\ q_{2,1}=\vartheta_1+\s_1+\f{2\s_1}{3}=\f43.
\end{equation}
So that \eqref{5.8} holds
and $q_{1,1}=4>\f{3p_0}{3-p_0}$.
Hence all the above calculations go through.

By inserting the Estimate \eqref{5.10} into \eqref{5.7}, with the indices given by \eqref{index0},
we obtain
\begin{equation}\label{5.12}\begin{split}
\frac{6}{11}\frac{d}{dt}\|W(t)\|_{L^{\frac{11}{6}}}^{\frac{11}{6}}+&
\frac{120}{121}\bigl\|\nabla|W|^{\frac{11}{12}}\bigr\|_{L^2}^2+\frac{72}{121}\int_{\R^3}\frac{|W|^{\frac{11}{6}}}{r^2}\,dx\\
\lesssim &\|r\ut\|_{L^{\f{11 p_0}{30(p_0-1)}}}^{\f{11}{30}}
\||W|^{\f{11}{12}}\|_{L^2}^{\f{3}{p_0}-\f{43}{20}}
\|\eta\|_{L^{p_0}}^{\frac{7p_0}{8}-\frac{1}{2}}\\
&\times
 \Bigl(\|\na|\eta|^{\f{p_0}2}\|_{L^{2}}^{2}+
\|\na|W|^{\f{11}{12}}\|_{L^2}^{2}+
\int_{\R^3}\f{|W|^{\frac{11}{6}}}{r^2}\,dx\Bigr).
\end{split}\end{equation}

Next,
applying $L^1$ energy estimate for the $\eta$ equation in \eqref{etaV} yields
\begin{equation}\label{5.13}
\frac{d}{dt}\|\eta(t)\|_{L^1}+\int_{\R^3}(-\Delta\eta)\cdot\rm{sgn}\ \eta\,dx
+2\int_{-\infty}^{+\infty}|\eta|\big|_{r=0} dz
\lesssim \bigl\|\pa_z|W|^{\frac {11}{12}}\bigr\|_{L^2}
\Bigl\|\frac{|W|^{\frac{13}{12}}}{r^{\frac{8}{11}}}\Bigr\|_{L^2}.
\end{equation}
Noting that $
\frac{|W|^{\frac{13}{12}}}{r^{\frac{8}{11}}}=\bigl|\frac{|W|^{\frac{11}{6}}}{r^2}\bigr|^{\f12}\cdot
\bigl|r^{\frac{18}{11}}W\bigr|^{\frac16}$,
so we have
\begin{equation}\label{5.15}
\Bigl\|\frac{|W|^{\frac{13}{12}}}{r^{\frac{8}{11}}}\Bigr\|_{L^2}
\leq \|r\ut\|_{L^\infty}^{\frac16}\cdot \Bigl\|\frac{|W|^{\frac{11}{6}}}{r^2}\Bigr\|_{L^1}^{\f12}.
\end{equation}
Substituting \eqref{5.15} into \eqref{5.13}, and using  the fact that
$\int_{\R^3}(-\Delta\eta)\cdot\rm{sgn}\ \eta\,dx\leq 0$, we achieve
\begin{equation}\begin{split}\label{5.16}
\frac{d}{dt}\|\eta(t)\|_{L^1}\lesssim&
\|r\ut\|_{L^\infty}^{\frac16} \bigl\|\pa_z|W|^{\frac {11}{12}}\bigr\|_{L^2} \Bigl\|\frac{|W|^{\frac{11}{6}}}{r^2}\Bigr\|_{L^1}^{\f12}\\
\lesssim& \|r\ut\|_{L^\infty}^{\frac16}
\Bigl(\bigl\|\pa_z|W|^{\frac {11}{12}}\bigr\|_{L^2}^2+\Bigl\|\frac{|W|^{\frac{11}{6}}}{r^2}\Bigr\|_{L^1}\Bigr).
\end{split}\end{equation}
Meanwhile, by taking $p=p_0$ in \eqref{3.20}, we have
\begin{equation}\begin{split}\label{5.17}
&\f{d}{dt}\bigl(\|V^{\e_0}(t)\|_{L^{q_0}}^{q_0}+\|\eta(t)\|_{L^{p_0}}^{p_0}\bigr)
+\frac{4(p_0-1)}{p_0}\|\nabla|\eta|^{\frac {p_0}2}\|_{L^2}^2\\
&\qquad\qquad+\frac{4(q_0-1)}{q_0}\|\nabla|V^{\e_0}|^{\frac {q_0}2}\|_{L^2}^2
+\frac{3p_0(-9p_0^2+21p_0-4)}{24p_0-2}\Bigl\|\frac{|V^{\e_0}|^{q_0}}{r^2}\Bigr\|_{L^1}\\
&\lesssim
\Bigl(\|r\ut_0\|_{L^{\alpha(p_0)}}^{\f{10(12p_0-1)}{3p_0(p_0+3)(3p_0+11)}}
\|V^{\e_0}\|_{L^2}^{q_0\cdot\f{23-17p_0-3p_0^2}{2p_0(3p_0+11)}}
\|\eta\|_{L^p_0}^{\frac{3p_0^2+23p_0-11}{2(3p_0+11)}}
+\|r\ut_0\|_{L^{\frac{2}{3(p_0-1)^2}}}^{\f1{6p_0}}
\|\eta\|^{\frac{p_0-1}4}_{L^{p_0}}\Bigr)\\
&\qquad\times \Bigl(\|\nabla|\eta|^{\frac {p_0}2}\|_{L^2}^2+\|\nabla|V^{\e_0}|^{\frac {q_0}2}\|_{L^2}^2
+\Bigl\|\frac{|V^{\e_0}|^{q_0}}{r^2}\Bigr\|_{L^1}\Bigr).
\end{split}\end{equation}
Summarizing the estimates \eqref{5.12}, \eqref{5.16} and \eqref{5.17} gives rise to
\begin{equation}\begin{split}\label{5.18}
&\f{d}{dt}\bigl(\|\eta(t)\|_{L^1}+\|W(t)\|_{L^{\frac{11}{6}}}^{\frac{11}{6}}
+\|V^{\e_0}(t)\|_{L^{q_0}}^{q_0}+\|\eta(t)\|_{L^{p_0}}^{p_0}\bigr)
+(p_0-1)\|\nabla|\eta|^{\frac {p_0}2}\|_{L^2}^2\\
&\qquad\qquad+\|\nabla|V^{\e_0}|^{\frac {q_0}2}\|_{L^2}^2
+\Bigl\|\frac{|V^{\e_0}|^{q_0}}{r^2}\Bigr\|_{L^1}
+\bigl\|\nabla|W|^{\frac{11}{12}}\bigr\|_{L^2}^2+\int_{\R^3}\frac{|W|^{\frac{11}{6}}}{r^2}\,dx\\
&\lesssim
\Bigl(\|r\ut_0\|_{L^{\alpha(p_0)}}^{\f{10(12p_0-1)}{3p_0(p_0+3)(3p_0+11)}}
\|V^{\e_0}\|_{L^2}^{q_0\cdot\f{23-17p_0-3p_0^2}{2p_0(3p_0+11)}}
\|\eta\|_{L^{p_0}}^{\frac{3p_0^2+23p_0-11}{2(3p_0+11)}}
+\|r\ut_0\|_{L^{\frac{2}{3(p_0-1)^2}}}^{\f1{6p_0}}
\|\eta\|^{\frac{p_0-1}4}_{L^{p_0}}\Bigr)\\
&\quad\times
\Bigl(\|\nabla|\eta|^{\frac {p_0}2}\|_{L^2}^2+\|\nabla|V^{\e_0}|^{\frac {q_0}2}\|_{L^2}^2
+\Bigl\|\frac{|V^{\e_0}|^{q_0}}{r^2}\Bigr\|_{L^1}\Bigr)\\
&\quad +\|r\ut_0\|_{L^\infty}^{\frac16}
\Bigl(\bigl\|\pa_z|W|^{\frac {11}{12}}\bigr\|_{L^2}^2+\Bigl\|\frac{|W|^{\frac{11}{6}}}{r^2}\Bigr\|_{L^1}\Bigr)
+\|r\ut_0\|_{L^{\f{11 p_0}{30(p_0-1)}}}^{\f{11}{30}}
\||W|^{\f{11}{12}}\|_{L^2}^{\f{3}{p_0}-\f{43}{20}}
\|\eta\|_{L^{p_0}}^{\frac{7p_0}{8}-\frac{1}{2}}\\
&\qquad\qquad\qquad\qquad\qquad\qquad\qquad
 \times\Bigl(\|\na|\eta|^{\f{p_0}2}\|_{L^{2}}^{2}+
\|\na|W|^{\f{11}{12}}\|_{L^2}^{2}+
\int_{\R^3}\f{|W|^{\frac{11}{6}}}{r^2}\,dx\Bigr).
\end{split}\end{equation}
Then under the smallness conditions \eqref{5.6}, and
\begin{equation}\label{5.19}
\|r\ut_0\|_{L^\infty}^{\frac16}\leq c_0,\quad
\|r\ut_0\|_{L^{\f{11 p_0}{30(p_0-1)}}}^{\f{11}{30}}
\||W|^{\f{11}{12}}(t_0)\|_{L^2}^{\f{3}{p_0}-\f{43}{20}}
\|\eta(t_0)\|_{L^{p_0}}^{\frac{7p_0}{8}-\frac{1}{2}}\leq c_0(p_0-1)
\end{equation}
 we get, by a standard continued argument, as we did in the last step of the proof of Proposition \ref{nearglobal}, that
\begin{equation}\label{5.20}\begin{split}
&\|\eta(t)\|_{L^1}+\|W(t)\|_{L^{\frac{11}{6}}}^{\frac{11}{6}}+\|V^{\e_0}(t)\|_{L^{q_0}}^{q_0}+\|\eta(t)\|_{L^{p_0}}^{p_0}\\
&\qquad +(p_0-1)\|\nabla|\eta|^{\frac {p_0}2}\|_{L^2}^2
+\int_{t_0}^t \int_{\R^3}\f{|W(t')|^{\frac{11}{6}}}{r^2}\,dxdt'\\
&\leq 2\bigl(\|\eta(t_0)\|_{L^1}+\|W(t_0)\|_{L^{\frac{11}{6}}}^{\frac{11}{6}}
+\|V^{\e_0}(t_0)\|_{L^{q_0}}^{q_0}+\|\eta(t_0)\|_{L^{p_0}}^{p_0}\bigr)\\
&\leq \frac{C}{t^{\frac{3(p_0-1)}{2}}}
\bigl(\|r\ut_0\|_{L^\infty}^{\frac{11}{6}}+L_{1,\frac{20}{3}}^{5p_0}+M_2^{3p_0}
+N^{5p_0}_{\frac{20}{13},20}\bigr),
\quad \forall\ t_0\leq t<\infty.
\end{split}\end{equation}
Recalling that $1<p_0<\min\bigl(1+\frac{1}{10A},\frac{21}{20}\bigr)$,
we have $\f{11 p_0}{30(p_0-1)}>\frac{11}{3}A$. Then it follows from Lemmas \ref{lem2.1.2} and  \ref{lem5.1}
that the condition \eqref{5.19} is verified  provided that
\begin{equation}\label{5.21}
\|r\ut_0\|_{L^\infty}\leq \frac{c_0}{L_{1,\frac{20}{3}}^6+M_2^2
+N^6_{\frac{20}{13},20}}
\Bigl((p_0-1)\bigl(t_0^{\frac{21}{16}p_0-\frac34}\|r\ut_0\|_{L^A}^{-A}\bigr)^{\frac{(p_0-1)}{p_0}}\Bigr)
^{\frac{240p_0}{300-127p_0-240A(p_0-1)}}.
\end{equation}

Finally we derive the $L^{\frac32}$ estimate for $U.$ Indeed
along the same line of the derivation of $\|W(t)\|_{L^{\f{11}6}},$
and using the indices given by \eqref{index0}, we infer
\begin{equation}\label{5.22}\begin{split}
\frac{d}{dt}\|U(t)\|_{L^{\frac32}}^{\frac32}+
\bigl\|\nabla|U|^{\frac34}\bigr\|_{L^2}^2
\lesssim&\int_{\R^3}\bigl|\f{u^r}r\bigr|\cdot|U|^{\frac32}\,dx\\
\lesssim& \|\eta\|_{L^{p_0}}^{\frac{7p_0}{8}-\frac12}
\bigl\|\na|\eta|^{\f{p_0}2}\bigr\|_{L^{2}}^{2\left(\frac{3}{2p_0}-\frac78\right)}\bigl\||U|^{\frac32}\bigr\|_{L^{\frac43}}.
\end{split}\end{equation}
whereas by applying H\"{o}lder's inequality, one has
\begin{equation}\label{5.23}
\begin{split}
\bigl\||U|^{\frac32}\bigr\|_{L^{\frac43}}=&
\Bigl(\int_{\R^3}|U|\Bigl(\f{|W|^{\frac{11}{6}}}{r^2}\Bigr)^{\frac25}
|r^2U|^{\f{4}{15}}\,dx\Bigr)^{\f34}\\
\lesssim& \|r\ut\|_{L^{\f{p_0}{5(p_0-1)}}}^{\f15}
\bigl\||U|^{\f 34}\bigr\|_{L^{\f{20p_0}{20-11p_0}}}
\Bigl(\int_{\R^3}\f{|W|^{\frac{11}{6}}}{r^2}\,dx\Bigr)^{\f{3}{10}},\\
\lesssim& \|r\ut\|_{L^{\f{p_0}{5(p_0-1)}}}^{\f15}
\bigl\||U|^{\f 34}\bigr\|_{L^2}^{\frac{3}{p_0}-\frac{43}{20}}
\bigl\|\nabla|U|^{\f 34}\bigr\|_{L^2}^{\frac{63}{20}-\frac{3}{p_0}}
\Bigl(\int_{\R^3}\f{|W|^{\frac{11}{6}}}{r^2}\,dx\Bigr)^{\f{3}{10}}.
\end{split}
\end{equation}
Substituting \eqref{5.23} into \eqref{5.22}, then using of Young's inequality gives rise to
\begin{equation*}\label{5.24}\begin{split}
\frac{d}{dt}\|U(t)\|_{L^{\frac32}}^{\frac32}+
\bigl\|\nabla|U|^{\frac34}\bigr\|_{L^2}^2
\lesssim& \|\eta\|_{L^{p_0}}^{\frac{7p_0}{8}-\frac12}
\|r\ut_0\|_{L^{\f{p_0}{5(p_0-1)}}}^{\f15}
\|U\|_{L^{\frac32}}^{\frac 32\bigl(\frac{3}{2p_0}-\frac{43}{40}\bigr)}\\
&\times\Bigl(\bigl\|\na|\eta|^{\f{p_0}2}\bigr\|_{L^{2}}^2
+\bigl\|\nabla|U|^{\f 34}\bigr\|_{L^2}^2+\int_{\R^3}\f{|W|^{\frac{11}{6}}}{r^2}\,dx\Bigr),
\end{split}\end{equation*} from which and \eqref{5.20}, we deduce by a
 standard continued argument that
\begin{equation}\begin{split}\label{5.25}
\|U(t)&\|_{L^{\frac32}}^{\frac32}+
\int_{t_0}^t \bigl\|\nabla|U|^{\frac34}(t')\bigr\|_{L^2}^2\,dt'\\
&\leq2\|U_0\|_{L^{\frac32}}^{\frac32}
+\frac{C}{t^{\frac{3(p_0-1)}{2}}}
\bigl(\|r\ut_0\|_{L^\infty}^{\frac{11}{6}}+L_{1,\frac{20}{3}}^{5p_0}+M_2^{3p_0}
+N^{5p_0}_{\frac{20}{13},20}\bigr),
\quad \forall\ t_0\leq t<\infty
\end{split}\end{equation}
provided that the smallness conditions \eqref{5.6},~\eqref{5.19},~\eqref{5.21} and
\begin{equation}\label{5.26}
\|\eta(t_0)\|_{L^{p_0}}^{\frac{7p_0}{8}-\frac12}
\|r\ut_0\|_{L^{\f{p_0}{5(p_0-1)}}}^{\f15}
\|U(t_0)\|_{L^{\frac32}}^{\frac 32\bigl(\frac{3}{2p_0}-\frac{43}{40}\bigr)}
\leq c_0(p_0-1),
\end{equation}
hold. Yet it follows  Lemma \ref{lem5.1} that \eqref{5.26} can be satisfied as long as
\begin{equation}\label{5.27}
\|r\ut_0\|_{L^\infty}\leq \frac{c_0}{L_{1,\frac{20}{3}}^{20}+M_2^{12}
+N^{20}_{\frac{20}{13},20}}
\Bigl((p_0-1)\cdot\bigl(t_0^{\frac{21}{16}p_0-\frac34}\|r\ut_0\|_{L^A}^{-A}\bigr)^{\frac{(p_0-1)}{p_0}}\Bigr)
^{\frac{5p_0}{p_0-5A(p_0-1)}}.
\end{equation}

Therefore under the smallness conditions \eqref{5.6},~\eqref{5.19},~\eqref{5.21} and  \eqref{5.27}, we get, by summing up \eqref{5.20} and \eqref{5.25}
that
for any $0\leq t<\infty $,
\begin{equation}\label{5.28}
\|\eta(t)\|_{L^1}+\|U(t)\|_{L^{\frac32}}^{\frac32}
\leq 2\|U_0\|_{L^{\frac32}}^{\frac32}
+\frac{C}{t^{\frac{3(p_0-1)}{2}}}
\bigl(\|r\ut_0\|_{L^\infty}^{\frac{11}{6}}+L_{1,\frac{20}{3}}^{5p_0}+M_2^{3p_0}
+N^{5p_0}_{\frac{20}{13},20}\bigr).
\end{equation}
This completes the proof of Theorem \ref{thmmain}
\end{proof}
\smallskip
\noindent {\bf Acknowledgments.}  P. Zhang is partially
supported by NSF of China under Grant 11371347 and innovation grant from National
Center for Mathematics and Interdisciplinary Sciences.

\end{document}